\documentclass[11pt]{amsart}
\usepackage[margin=1in]{geometry}
\usepackage{amsmath, amssymb, amsthm, url, stmaryrd, cancel, mathtools, scalerel}
\usepackage{enumitem}
\usepackage{tikz}
\usetikzlibrary{positioning}
\usepackage{url}

\setlist[enumerate]{nosep}

\numberwithin{equation}{section}
\newlength{\listeria}
\theoremstyle{newstyle}
\newtheorem{thm}{Theorem}

\theoremstyle{plain}
\newtheorem{prop}{Proposition}[section]
\newtheorem{lemma}[prop]{Lemma}
\newtheorem{cor}[prop]{Corollary}

\newtheorem{thmA}{Theorem}

\theoremstyle{definition}
\newtheorem*{defn}{Definition}
\newtheorem{fact}{Fact}[section]

\newtheorem{example}{Example}[section]
\newtheorem*{example*}{Example}

\theoremstyle{remark}

\newtheorem*{rk}{Remark}

\newcommand{\into}{\hookrightarrow}
\newcommand{\onto}{\twoheadrightarrow}

\newcommand{\Z}{\mathbb{Z}}

\newcommand{\Q}{\mathbb{Q}}
\newcommand{\R}{\mathbb{R}}
\newcommand{\N}{\mathbb{N}}

\newcommand{\fp}{{\mathbb{F}_{p}}}

\newcommand{\F}{{\mathbb{F}}}

\newcommand{\fpbar}{\algebraiclosure{\F}_p}

\renewcommand{\tilde}{\widetilde}
\newcommand{\mm}{{\mathfrak m}}

\newcommand{\overbar}[1]{\mkern 1.5mu\overline{\mkern-3mu#1\mkern-1.5mu}\mkern 1.5mu}
\newcommand{\algebraiclosure}[1]{\overbar{{#1}}}

\newcommand{\Kbar}{\algebraiclosure{K}}

\DeclareMathOperator{\eend}{{\rm End}}

\DeclareMathOperator{\cchar}{\rm char}

\newcommand{\fromlim}{\varprojlim}

\newcommand{\jb}{Bella\"iche}

\newcommand{\gouvea}{Gouv\^ea}

\newcommand{\dd}{\Delta}

\newcommand{\lb}{\llbracket}

\newcommand{\rb}{\rrbracket}

\DeclareMathOperator{\cmod}{\,{\rm mod}}

\DeclareMathOperator{\rite}{\rho}
\DeclareMathOperator{\reed}{\pi}
\DeclareMathOperator{\cc}{\tilde {\it c}}

\usepackage{hyperref}
\hypersetup{
    colorlinks=true,
    linkcolor=blue,
    urlcolor = cyan, 
    citecolor = blue}

\begin{document}

\title{Nilpotence order growth of recursion operators in characteristic $p$}
\author{Anna Medvedovsky}
\address{Max-Planck-Institut f\"ur Mathematik, Vivatsgasse 7, 53111 Bonn Germany}
\email{medved@mpim-bonn.mpg.de}

\begin{abstract}We prove that the killing rate of certain degree-lowering ``recursion operators" on a polynomial algebra over a finite field grows slower than linearly in the degree of the polynomial attacked. We also explain the motivating application: obtaining a lower bound for the Krull dimension of a local component of a big mod-$p$ Hecke algebra in the genus-zero case. We sketch the application for $p = 2$ and $p = 3$ in level one. The case $p = 2$ was first established in by Nicolas and Serre in 2012 using different methods.
\end{abstract}

\maketitle


\setcounter{tocdepth}{1}
\tableofcontents


\setlength\abovedisplayshortskip{0pt}
\setlength\belowdisplayshortskip{0pt}
\setlength\abovedisplayskip{5pt}
\setlength\belowdisplayskip{5pt}

\section{Introduction}
The main goal of this document is to prove the following Nilpotence Growth Theorem, about the killing rate of a recursion operator on a polynomial algebra over a finite field under repeated application: 

\begin{thmA}[Nilpotence Growth Theorem; see also Theorem \ref{ngtthm}]\label{mainthm}
Let $\F$ be a finite field, and suppose that $T: \F[y] \to \F[y]$ is a degree-lowering $\F$-linear operator satisfying the following condition:
\begin{enumerate}
\item[]
The sequence $\{T(y^n)\}_n$ of polynomials in $\F[y]$ satisfies a linear recursion over $\F[y]$ whose companion polynomial\footnote{See section \ref{sec:linreccompoly} for definitions.}
$X^d + a_1 X^{d-1} + \cdots + a_d \in \F[y][X]$ has both total degree $d$ and $y$-degree $d$. 
\end{enumerate}
Then there exists a constant $\alpha < 1$ so that the minimum power of $T$ that kills $y^n$ is $O(n^\alpha)$. 
\end{thmA}
To prove this theorem, we reduce to the case where the companion polynomial of the recursion has an ``empty middle" in its degree-$d$ homogeneous part: that is, when it has the form $X^d + a y^d + (\mbox{lower-order terms})$ for some $a \in \F$. Then we prove this empty-middle case (see Theorem \ref{specialngtthm} below) by constructing a function $c: \F[y] \to \N \cup \{-\infty\}$ that grows like $(\deg f)^\alpha$ and whose value is lowered by every application of $T$. In the special case where $d$ is a power of $p$, the function $c$ takes $y^n$ to the integer obtained by writing $n$ base $d$ and then reading the expansion in some smallers base, so that the sequence $\{c(y^n)\}_n$ is $p$-regular in the sense of Allouche and Shallit \cite{alloucheshallit}. The proof that $c(T(y^n)) < c(y^n)$, by strong induction, uses higher-order recurrences depending on $n$, so that $n$ is compared to numbers whose base-$d$ expansion is not too different. 

It is the author's hope that ideas from $p$-automata theory can eventually be used to sharpen and generalize the Nilpotence Growth Theorem.

{\bf Motivating application of the Nilpotence Growth Theorem:} The motivating application for the Nilpotence Growth Theorem (Theorem \ref{mainthm} above) is the so-called \emph{nilpotence method} for establishing lower bounds on dimensions of local components of Hecke algebras acting on mod\nobreakdash-$p$ modular forms of tame level~$N$. These Hecke algebra components were first studied by Jochnowitz  in the 1970s \cite{JochStudy}, but the first full structure theorem, for $p = 2$ and $N =1$, due to Nicolas and Serre, appeared only in 2012 \cite{NS2}. The Nicolas-Serre method uses the recurrence satisfied by Hecke operators (see equation \eqref{eq:rec2} below) to describe the action of Hecke operators on modular forms modulo $2$ completely explicitly---but unfortunately these explicit formulas do not appear to generalize beyond $p = 2$. The case $p \geq 5$ was then established by very different techniques by \jb\ and Khare for $N = 1$ \cite{BK}, and later by Deo for general level~\cite{deo}. The \jb-Khare method deduces information about  mod-$p$ Hecke algebra components from corresponding characteristic-zero Hecke algebra components, which are known to be big by the \gouvea-Mazur ``infinite fern" construction (\cite{infinitefern}; see also \cite[Corollary 2.28]{emertonBBK}). The nilpotence method is yet a third technique, coming out of an idea of \jb\ for tackling the case $p = 3$ and $N = 1$ as outlined in \cite[appendix]{BK}, and implemented and developed in level one for $p= 2, 3, 5, 7, 13$ in the present author's  Ph.D. dissertation~\cite{medved}. Like Nicolas-Serre, the nilpotence method stays entirely in characteristic $p$ and makes use of the Hecke recursion; but instead of explicit Hecke action formulas, the {Nilpotence Growth Theorem} (Theorem \ref{mainthm} above) now plays the crucial dimension-bounding role. See section \ref{appstoheckesec} below for a taste of this method for $p = 2, 3$, which completes the determination of the structure of the Hecke algebra for $p = 3$ begun in \cite[appendix]{BK} and recovers the Nicolas-Serre result for $p = 2$.
The nilpotence method using the {Nilpotence Growth Theorem} can be generalized for all $(p, N)$ if the genus of $X_0(Np)$ is zero; see the forthcoming \cite{medved:heckedim} for details.

{\bf Structure of this document:} After a few preliminary definitions in section \ref{companionpolysec}, we state the Nilpotence Growth Theorem (restated as Theorem \ref{ngtthm}) in section \ref{ngtstatesec} and discuss the various conditions of the theorem. In section \ref{toyngtsec}, we prove a toy version of the theorem (Theorem \ref{toyngtthm}). In section \ref{appstoheckesec}, we use the toy version of NGT (Theorem~\ref{toyngtthm}) from the previous sections to prove that the mod-$p$ level-one Hecke algebra for $p = 2, 3$ has the form $\fp\lb x, y \rb$. This section illustrates the motivating application of the Nilpotence Growth Theorem, and is not required for the rest of the document. This would be a reasonable stopping point for a first reading. 

At this point, in section \ref{proofbeginsec} the proof begins in earnest. There is a short overview of the structure of the proof in subsection~\ref{overproofsec}. In subsection~\ref{reducetoFsec}, we reduce to working over a finite field. In subsection~\ref{specialngtsec}, we reduce to the so-called empty-middle NGT (Theorem \ref{specialngtthm}). In subsection~\ref{ngtpf}, we give the inductive argument that reduces the proof of the empty-middle NGT to finding a \emph{nilgrowth witness} function that satisfies certain properties. In section \ref{contsec}, we discuss base-$b$ representation of numbers and introduce the content function. In section \ref{sec:ineqs}, we prove a number of technical inequalities about the content function. In section \ref{nilgrowthsec} we define the nilgrowth witness, finishing  the proof of the empty-middle NGT, and hence of the NGT in full. Finally in section~\ref{complementsec}, we state a more precise version of the toy NGT and speculate on the optimality of some bounds. 

{\bf Acknowledgements:} I am greatly indebted to my Ph.D. advisor Jo\"el \jb\ for sparking and supporting this investigation into modular forms modulo $p$. The initial idea for the nilpotence method is his---a most generous gift. I would also like to thank Paul Monsky for many illuminating discussions on the topic of mod-$p$ modular forms and Hecke algebras, and Kiran Kedlaya for helpful comments on an earlier version of Theorem \ref{ngtthm}.  Part of the writing of this document was carried out at the Max Planck Institute for Mathematics in Bonn, and I am grateful for their hospitality. Many extensive computations related to this project were performed using {\tt SAGE}~\cite{sage}.

\section{Preliminaries}\label{companionpolysec}

This section contains a brief review of a few unconnected algebraic notions.  All rings and algebras are assumed to be commutative, with unity. We use the convention that the set of natural numbers starts with zero: $\N = \{0, 1, 2, \ldots\}$. Below, $R$ is always a ring. 

\subsection{Structure of finite rings}
If $R$ finite, then $R$ is artinian, hence a finite product of finite local rings. If $R$ is a finite local ring with maximal ideal $\mm$, then the residue field $R/\mm$ is a finite field of characteristic $p$. Moreover, the graded pieces $\mm^{n}/\mm^{n + 1}$ are finite $R/\mm$-vector spaces, so that that $R$ has cardinality a power of $p$. Basic examples of finite local rings are $\fp[t]/(t^k)$ and $\Z/p^k\Z$.

\subsection{Degree filtration on a polynomial algebra}
If $0 \neq f= \sum_{n \geq 0} c_n y^n$ is a polynomial in $R[y]$ (so only finitely many of the $c_n$ are nonzero), then its \emph{$y$-degree}, or just \emph{degree}, is as usual defined to be $\deg f := \max \{n: c_n \neq 0\}$. Also let $\deg 0 : = -\infty$. 

The degree function gives $R[y]$ the structure of a \emph{filtered algebra}: Let $R[y]_n:= \{f \in R[y]: \deg f \leq n\}$, and then $R[y] = \bigcup_{n \geq 0} R[y]_n$ and multiplication preserves the filtration as required.

\subsection{Local nilpotence and the nilpotence index}

Let $M$ be any $R$-module (for example, $M = R[y]$) and $T \in \eend_R(M)$ an $R$-linear endomorphism. (In applications to Hecke algebras, $R$ will be a finite field, $M$ an infinite-dimensional space of modular forms, and $T$ a Hecke operator.) The operator $T: M \to M$ is \emph{locally nilpotent} on $M$ if every element of $M$ is annihilated by some power of $T$. 
If $T$ is locally nilpotent, and $f$ in $M$ is nonzero, we define the \emph{nilpotence index of $f$ with respect to $T$}: 
$$N_T(f) := \max \{k \geq 0: T^k f \neq 0\}.$$
Also set $N_T(0) : = -\infty$. 

Suppose $R = K$ is a field and $M = K[y]$ and $T: M \to M$ preserves the degree filtration: that is, $T(K[y]_n) \subset K[y]_n$. Then $T$ is locally nilpotent if and only if $T$ strictly lowers degrees, in which case we also have $N_T(f) \leq \deg f$. 

For example, $T = \frac{d}{dy}$ is locally nilpotent on $K[y]$. If $K$ has characteristic zero, then $N_T(f) = \deg f$; otherwise $N_T(f) \leq \cchar K-1$. 

\subsection{Linear recurrences and companion polynomials}\label{sec:linreccompoly}
Now suppose that $M$ is an $R$-algebra, and $M'$ is an $M$-module (we will normally take $M = M'$). A sequence $s = \{s_n\}_n \in M'^\N$ satisfies a (monic) \emph{$M$-linear recurrence of order $d$} if there exist elements $a_1, \ldots, a_d \in M$ so that  
\begin{equation}\label{eq:linrec}
s_n = a_1 s_{n-1} + \cdots + a_d s_{n - d}
\qquad \mbox{for all $n \geq d$}.
\end{equation} 
We do not a priori assume that $a_d \neq 0$, but we do insist that the recursion already hold for $n = d$. 
The \emph{companion polynomial} of this linear recurrence is $P(X) = X^d - a_1 X^{d - 1} - \cdots - a_d \in M[X].$ 

\begin{example}\label{trivialcounterex}
The sequence $s = \{0, 1, y, y^2, y^3, y^4,\ldots\} \in R[y]^\N$ satisfies an $R[y]$-linear recursion of minimal order~2: we have $s_n = y s_{n - 1}$ for all $n \geq 2$, but not for $n = 1$. The companion polynomial of the recurrence is therefore $X^2 - yX$. 
\end{example}

Given any sequence $s$ in $M'^\N$, the set of companion polynomials of $M$-linear recurrences satisfied by $s$ forms an ideal of $M[X]$. We record this observation in the following form:  

\begin{fact}\label{fact:compolyideal}
If a sequence $s \in M'^\N$ satisfies the recurrence defined by some monic $P \in M[X]$, then it also satisfies the recurrence defined by $PQ$ for any other monic $Q \in M[X]$. 
\end{fact}

In characteristic $p$, then, we get the following corollary, of which we will make crucial use: 

\begin{cor}\label{cor:ppowercharp}
If $R$ has characteristic $p$, and $s \in M^\N$ satisfies the order-$d$ recurrence
$$s_n = a_1 s_{n-1} + a_2 s_{n -2} + \cdots + a_d s_{n-d} \quad \mbox{for all $n \geq d$},$$
then for every $k \geq 0$ the sequence $s$ also satisfies the order-$dp^k$ recurrence
\begin{equation}\label{eq:powerp}
s_n = a_1^{p^k} s_{n - p^k} + a_2^{p^k} s_{n - 2 p^k} + \cdots +  a_d^{p^k} s_{n - d p^k} \quad 
\mbox{for all $n \geq dp^k$}.
\end{equation}
\end{cor}

\begin{proof}
Let $P = X^d - a_1 X^{d-1} - \cdots - a_d$ be the companion polynomial of a recursion satisfied by $s$. By Fact~\ref{fact:compolyideal}, the sequence $s$ also satisfies the recurrence whose companion polynomial is 
$$P^{p^k} = X^{d p^k} - a_1^{p^k} X^{dp^k - p^k} - a_2^{p^k} X^{d p^k - 2p^k} - \cdots - a_d^{p^k},$$
which is exactly what is expressed in equation \eqref{eq:powerp}. 
\end{proof}

If $M$ can be embedded into a field $K$, we have the following well-known characterization of power sequences in $\Kbar^\N$ satisfying a fixed $M$-linear recurrence: 
\begin{fact}
An element $\alpha$ in $\Kbar$ is a root of monic $P \in M[X]$ if and only if the sequence $\{\alpha^n\}_{n} = \{1, \alpha, \alpha^2, \ldots\}$ satisfies the linear recurrence with companion polynomial $P$.
\end{fact}

If the companion polynomial of such an $M$-linear recurrence has no repeated roots in $\Kbar$, it follows from the proposition that \emph{every} solution to the recurrence is a linear combination of such power sequences on the roots of the companion polynomial. One can further describe all $\Kbar$-sequences satisfying a general $M$-linear recursion --- see, for example, \cite{kconrad:recursion} and the historical references therein --- but we will not need this below.

\section{The Nilpotence Growth Theorem (NGT)}\label{ngtstatesec}

\subsection{Statement of the NGT}\label{ngtstatesubsec}

We are now ready to state the most general version of the Nilpotence Growth Theorem (NGT). From now on, we will assume that $R$ will be a \emph{finite} ring, and $M = R[y]$. The eventual Hecke examples will come from the case where $R$ is a finite field.

\begin{thm}[Nilpotence Growth Theorem]\label{ngtthm}
Let $R$ be a finite ring, and suppose that $T: R[y] \to R[y]$ is an $R$-linear operator satisfying the following two conditions:
\begin{enumerate}
\item\label{Tdeglow} $T$ lowers degrees: $\deg T(f) < \deg f$ for every nonzero $f$ in $R[y]$.
\item\label{Trecop} The sequence $\{T(y^n)\}_n$ satisfies a \emph{filtered} linear recursion over $R[y]$: that is, there exist $a_1, \ldots, a_d \in R[y]$, with $\deg a_i \leq i$ for each $i$, so that for all $n \geq d$, $$T(y^n) = a_1 T(y^{n-1}) +  \cdots + a_d T(y^{n-d}).$$
\end{enumerate}

Suppose further that  
\begin{enumerate}[resume]
\item \label{T3} the coefficient of $y^d$ in $a_d$ is invertible in $R$.
\end{enumerate} 

Then there exists a constant $\alpha < 1$ so that $N_T(y^n) \ll n^\alpha$. 
\end{thm}

In other words, Theorem \ref{ngtthm} implies that, under a mild technical assumption (condition (3)), the nilpotence index of a degree-lowering operator defined by a filtered linear recursion grows \emph{slower than linearly} in the degree. The mild technical assumption is necessary in the theorem as stated: see the discussion in \eqref{T3nec} in subsection \ref{ngtremarksec} below.

\subsection{Discussion of the NGT}\label{ngtremarksec}
\begin{enumerate}[leftmargin=*,align=left]

\item{\bf Connection with Theorem \ref{mainthm}:} If $T: R[y] \to R[y]$ satisfies the conditions of Theorem~\ref{ngtthm}, then the companion polynomial of the recursion satisfied by the sequence $\{T(y^n)\}_n$ is $$P_T = X^d - a_1 X^{d-1} - \ldots - a_d \in R[y, X].$$ The condition $\deg a_i \leq i$ from (\ref{Trecop}) guarantees that the total degree of $P_T$ is exactly $d$. In particular, in the case where $R= \F$ is a finite field, condition~(\ref{T3}) implies that $\deg_y P_T = \deg a_d = d$. In other words, Theorem~\ref{ngtthm} over a finite field reduces to Theorem~\ref{mainthm}, as implied.  

\item{\bf Condition (\ref{Tdeglow}) guarantees that $T$ is locally nilpotent.} 
Moreover, $N_T(y^n) \leq n$, so that the function $n \mapsto N_T(y^n)$ a priori grows no faster than linearly.

\item {\bf Condition (\ref{Trecop}) and connection to recursion operators:} The condition that the sequence $\{T(y^n)\}_n$ satisfies a linear recurrence is the definition of a \emph{recursion operator}, a notion that will be explored in a future paper. A natural source of \emph{filtered} recursion operators (that is, satisfying additional degree bounds as in condition (\ref{Trecop}) above) comes from the action of Hecke operators on algebras of modular forms of a fixed level. Namely, \emph{if $f$ is a modular form of weight $k$ and level $N$ and $T$ is a Hecke operator acting on the algebra $M$ of forms of level $N$, then the sequence $\{T(f^n)\}_n$ satisfies an $M$-linear recursion with companion polynomial $X^d + a_1 X^{d-1} + \cdots + a_d$, where $a_i$ comes from weight $ki$.} 

\noindent See equations \eqref{eq:rec2} and \eqref{eq:rec3} below for examples over $\fp$, \cite[chapter 6]{medved} for a proof in level one when $T$ is a prime Hecke operator, or the forthcoming \cite{medved:heckedim} for more details. This kind of recurrence in characteristic $2$ was crucially used in Nicolas-Serre \cite[section 3]{NS1} to obtain the structure of the mod-$2$ Hecke algebra. Earlier appearances of the Hecke recurrence include~\cite{kamal} (two-dimensional recurrence for $\{T_\ell(E_4^n E_6^m)\}_{n, m}$ in characteristic zero) and \cite[p.\ 594]{BuzzCale:slopes} (recurrence for $U_2$ acting on a power basis of overconvergent $2$-adic modular functions). 

\item\label{T3nec} {\bf Condition (\ref{T3}) is necessary as stated:} Consider the operator $T: R[y] \to R[y]$ defined by $T(y^n) = s_n$, where $s_n$ is the sequence $\{0, 1, y, y^2, \ldots\}$ with companion polynomial $X^2 - yX$ from Example \ref{trivialcounterex}. All conditions except~ (\ref{T3}) are satisfied, and it is easy to see that $N_T(y^n) = n$ in this case.
 
For an example with $a_d \neq 0$, consider the operator $T$ with defining companion polynomial $P_T = X^2 + y X + y$ with initial values $[T(1), T(y)] = [0, 1]$. By induction, $\deg T(y^n) = n - 1$. Therefore $N_T(y^n) = n$. 

Computationally, it appears that if $R = \fp$ and $\deg a_d < d$ but there exists an $i$ with $0 < i < d$ so that $\deg a_i = i$, then either $N_T$ grows logarithmically or else it grows linearly. In that sense, it appears that ``fullness" of degree at the end of $P_T$ (that is, the presence of a $y^d$ term) appears to be, at least generically, necessary to compensate for ``fullness" of degree in the middle (that is, the presence of a $y^i X^{d-i}$ term for some $0< i < d$), if one wants the growth of $N_T$ to be sublinear but not degenerate. Further explanation is necessary to understand this behavior well. 

\item {\bf The constant $\alpha$:} The power $\alpha$ depends on $R$ and $d$ only, and tends to $1$ as $d \to \infty$. More precisely, the dependence on $R$ is only through its maximal residue characteristic; the length of $R$ as a module over itself affects only the implicit constant of the growth condition $N_T(y^n) \ll n^\alpha$. In the special case empty-middle case where the inequality $\deg a_i < i$ is strict for every $i < d$, we can take $\alpha$ to be $\log_{p^k}(p^k - 1)$ for $k$ satisfying $d \leq p^k$. See Theorem~\ref{toyngtthm} or Theorem \ref{specialngtthm} below.

\item {\bf Finite characteristic is necessary:} A counterexample in characteristic zero: Consider the operator $T$ on $\Q[y]$ with  $P_T = X^2 - y X - y^2$ and degree-lowering initial terms $[T(1), T(y)] = [0, 1]$. This satisfies all three conditions of the NGT. It is easy to see that $T(y^n) = F_n y^{n-1}$, where $F_n$ is the $n^{\rm th}$ Fibonacci number: the recursion is $s_n = y s_{n-1} + y^2 s_{n-2}$. Therefore $$T^k(y^n) = F_n F_{n-1} \cdots F_{n-k + 1} \, y^{n-k},$$ so that $N_T(y^n)= n$. (Compare to characteristic $p$, where the operator defined by $T(y^n) = F_n y^{n-1}$ on $\fp[y]$ satisfies $T^{p + 1} \equiv 0$.) See also Proposition \ref{prop:charzeroex} for a family of examples in any degree. 

Computationally, it appears that generic characteristic-zero examples that do not degenerate (to $\log n$ growth) all exhibit linear growth. Over a finite field, computationally one sees a spectrum of $O(n^\alpha)$ growth for various $\alpha < 1$. 

\item {\bf Finiteness of $R$ is necessary as stated:} A counterexample over $\fp(t)$: Let $P_T = X^2 - t y X - y^2$ and start with $[0, 1]$ again. Then $T(y^n) = F_n(t) y^{n-1}$ with $F_n(t) \in \fp[t]$ monic of degree $n-1$, so that $N_T(y^n) = n$ again.
However, see the empty-middle case (Theorem \ref{specialngtthm}) for a special case that does hold for infinite rings of characteristic $p$.
\end{enumerate}

\section{A toy case of the NGT}\label{toyngtsec}


Fix a prime $p$ and take $R = \fp$\footnote{In fact $R$ may be any ring of characteristic $p$. We work with $\fp$ in this section for simplicity.},
$d = p^k$, and assume that the recursion has an ``empty middle": $\deg a_i < i$ for $0 < i < d$.

\begin{thm}[Toy case of NGT]\label{toyngtthm}
Let $q = p^k$ for some $k \geq 1$. Suppose $T: \fp[y] \to \fp[y]$ is a degree-lowering linear operator so that the sequence $\{T(y^n)\}_n$ satisfies an $\fp[y]$-linear recursion with companion polynomial $$P = X^q + (\mbox{terms of total degree} < q) + ay^q \quad \in \fp[y][X]$$
for some $a \in \fp$. Then $N_T(y^n) \ll n^{\log (q-1)/\log q}.$
\end{thm}
Most of the main features of the proof of Theorem~\ref{ngtthm} are already present in the proof of Theorem~\ref{toyngtthm}. We give it here because the proof is technically much simpler; understanding it may suffice for all but the most curious readers. 

\subsection{The content function}
Following \jb\ in \cite[appendix]{BK} for $q = 3$, we define a function $c : \N \to \N$ depending on $q$ as follows. Given an integer $n$, we write it base $q$ as $n = \sum_i n_i q^i$ with $0 \leq n_i < q$, only finitely many of which are nonzero, and define the \emph{$q$-content of $n$} as
$c(n) := c_q(n) := \sum_i n_i (q - 1)^i.$ 
For example, since $71 = [2\,4\, 1]_5$ in base $5$, the $5$-content of 71 is $2\cdot4^2 + 4\cdot4 + 1 = 49.$

The following properties of the content function are easy to check. See also section \ref{contsub}, where the content function and variations are discussed in detail. 
\begin{prop}\label{babycontent}
\begin{enumerate}\vspace{-\listeria}
\item \label{babygrowth} $c(n) \ll n^{\log_q(q-1)}$
\item \label{babyscale} $c(q^k n) = (q-1)^k c(n)$ for all $k \geq 0$ 
\item \label{babybase} If $0 \leq n < q$, then $c(n) = n$.
\item \label{iDitem} If $i$ is a digit base $q$ and $n \geq i$ has no more than $2$ digits base $q$, then $c(n) - c(n - i)$ is either $i$ or $i - 1$. 
\item \label{iditem} If $q \leq n < q^2$, then $c(n - q) = c(n) - q + 1$.
\end{enumerate} 
\end{prop}

\subsection{Setup of the proof}

We now define the $q$-content of a polynomial $f \in \fp[y]$ through the $q$-content of its degree. More precisely, if $0 \neq f = \sum a_n y^n$, let $\cc(f) := \max \{c(n): a_n \neq 0\}$. Also set $\cc(0) := -\infty$. For example, the $3$-content of $2 y^9 + y^7 + y^2$ is $\max\{c(9), c(7), c(2)\} = 5$. 

Now let $T: \fp[y] \to \fp[y]$ be a degree-lowering recursion operator whose companion polynomial 
$$P = X^q + a_1(y) X^{q - 1} + \cdots + a_q(y) \in \fp[y][X]$$
satisfies $\deg a_i(y) < i$ for $1 \leq i < q$ and $\deg a_q(y) \leq q$.

To prove Theorem \ref{toyngtthm}, we will show that $T$ lowers the $q$-content of any $f \in \fp[y]$: that is, that $\cc(Tf) < \cc(f)$. Since $\cc(f) < 0$ only if $f = 0$, the fact that $T$ lowers $q$-content implies that $N_T(f) \leq c(f)$. Proposition~\ref{babycontent}\eqref{babygrowth} will then imply $N_T(f) \ll (\deg f)^{\log(q-1)/\log q}$, as desired.

It suffices to prove that $\cc(Tf) < \cc(f)$ for $f = y^n$. We will proceed by strong induction on $n$, each time using recursion of order $q^{k + 1}$ corresponding to $P^{q^k}$, with $k$ chosen so that $q^{k+1} \leq n < q^{k + 2}$. This will allow us to compare the $q$-contents of $n$ and $n - iq^k$ for $i$ small, so that the last $k$ digits base $q$ are unchanged, rather than comparing $q$-contents of $n$ and $n - i$ for $i$ small, which can be very destructive to digits base $q$.\footnote{I learned this technique from Gerbelli-Gauthier's proof \cite{mathilde} of the key technical lemmas of Nicolas-Serre \cite{NS1}.} The base case is $n < q$, in which case being $q$-content-lowering is the same thing as being degree-lowering (Proposition \ref{babycontent}\eqref{babybase}).

\subsection{The induction}

For $n \geq q$, we must show that $\cc(T(y^n)) < c(n)$ assuming that $\cc(T(y^m)) < c(m)$ for all $m < n$.  As above, choose $k \geq 0$ with $q^{k + 1} \leq n < q^{k + 2}$. By Corollary \ref{cor:ppowercharp}, the sequence $\{T(y^n)\}$ satisfies the order-$q^{k + 1}$ recurrence  
$$T(y^{n}) =  a_1(y)^{q^k} \,T(y^{n - q^k}) + a_2(y)^{q^k} \,T(y^{n - 2q^k}) + \cdots + a_q(y)^{q^k}\, T(y^{n - q^{k + 1}}).$$
Pick a term $y^m$ appearing in $T(y^n)$ with nonzero coefficient; we want to show that $c(m) < c(n)$. From the recursion, $y^m$ appears with nonzero coefficient in $a_i(y)^{q^k} T(y^{n - i  q^k})$ for some $i$. More precisely, $y^m$ appears in $y^{j q^k} T(y^{n - i q ^k}) $ for some $y^j$ appearing in $a_i(y)$, so that either $j < i$ or $i = j = q$.
Then $y^{m - j q^k}$ appears in $T(y^{ n - i q^k})$, and by induction we know that $c(m - j q^k) <  c(n - i q^k)$. To conclude that $c(m) < c(n)$, it would suffice to show that 
$$c(n) - c(m) \geq c(n - i q^k) - c(m - j q^k),$$ or, equivalently, that
$$c(n) - c(n - i q^k) \geq c(m) - c(m - j q^k).$$

Since subtracting multiples of $q^k$ leaves the last $k$ digits of $n$ base $q$ untouched, we may replace $n$ and $m$ by $q^k\lfloor \frac{n}{q^k} \rfloor$ and $q^k\lfloor \frac{m}{q^k} \rfloor$, respectively, and then use Proposition \ref{babycontent}\eqref{babyscale} to cancel out a factor of $(q-1)^k$. In other words, we must show that  
$$c(n) - c(n - i) \geq c(m) - c(m -j)$$
for $n$, $m$, $i$, $j$ satisfying $i \leq n < q^2$ and $j \leq m < n$ and either $j < i$ or $i = j = q$. 
But this is an easy consequence Proposition \ref{babycontent}\eqref{iDitem}--\eqref{iditem}: For $j < i$, we know that $c(n) - c(n - i)$ is at least $i - 1$ and $c(m) - c(m -j)$ is at most $j \leq i - 1$. And for $i = j = q$ both sides equal $q - 1$. 

This completes the proof of Theorem \ref{toyngtthm}. 

\subsection{Toy case vs. general case}

 The proof of the full NGT (Theorem \ref{ngtthm}) proceeds by first reducing to the empty-middle case over a finite field (Theorem \ref{specialngtthm} below), of which Theorem~\ref{toyngtthm} is a special case. Apart from the reduction step, most of the difficulty in generalizing from Theorem~\ref{toyngtthm} to Theorem \ref{specialngtthm} comes from working with more general versions of the content function. In particular, we must extend the notion of content to rational numbers and prove sufficiently strong analogues of Proposition \ref{babycontent} for the proof to proceed: see sections \ref{contsec}--\ref{nilgrowthsec}.
\section{Applications to mod-$p$ Hecke algebras for $p = 2, 3$} \label{appstoheckesec}

This section gives an indication of how the NGT can give information about lower bounds of mod-$p$ Hecke algebras, the author's main motivation for proving the theorem. More precisely, in this section we will use Theorem \ref{toyngtthm} to complete the proof of Theorem 24\footnote{We prove a weaker version of Proposition 35 {\it loc.\ \!cit.} In the notation of subsection \ref{contentdefsec} here, we show that, for $f \in \F_3[\dd]$, we have $c_{3, 2}(T_2 f) \leq c_{3, 2}(f)-1$ and that $c_{9, 6}(T_7'f) \leq c_{9, 6}(f) -3$, which suffice to establish that the nilpotence index grows slower than linearly. We do not show that $c_{3, 2}( T_7'f) \leq c_{3, 2}(f) - 2$.} of \cite[appendix]{BK}, which establishes the structure of the mod-$3$ Hecke algebra of level one. Simultaneously and using the same methods, we will give an alternate proof of the main result of \cite{NS2}, the structure of the mod-$2$ Hecke algebra in level one. See Theorem \ref{heckeappthm} below. 

More generally, the NGT can be used to obtain lower bounds on Krull dimensions of local components of big mod-$p$ Hecke algebras acting on forms of level $N$ in the case where $X_0(Np)$ has genus zero, for this is precisely the condition for the algebra of modular forms of level $N$ mod $p$ to be a polynomial algebra over $\fp$. For more details, see \cite{medved} (for level one) or \cite{medved:heckedim}, to appear in the future. To generalize the nilpotence method to all $(p, N)$, one must generalize the NGT to all rings of $S$-integers in characteristic-$p$ global fields. 

We work in level one, and let $p \in \{2, 3\}$. 

Let $M = M(1, \fp) \subset \fp\lb q \rb$ be the space of modular forms of level one modulo $p$ in the sense of Swinnterton-Dyer and Serre (that is, reductions of integral $q$-expansions). For $p = 2, 3$, Swinnterton-Dyer observes \cite{SwDy} that $M = \fp[\dd]$, where $\dd = \prod_{i = 1}^n (1 - q^n)^{24} \in \fp \lb q \rb$; standard dimension formulas show that $M_k := \fp[\dd]_k$, the polynomials in $\dd$ of degree bounded by $k$, coincides with the space of mod-$p$ reductions of $q$-expansions of forms of weight $12k$, and hence Hecke-invariant. Further, one can show that $K =  \big\langle \dd^n : p \nmid n  \big\rangle_{\fp} \subset M$
is the kernel of the operator $T_p$. 
In particular, $K$, and hence every finite-dimensional subspace $K_k := M_k \cap K$, is Hecke-invariant. 

Let $A_k \subset \eend_\fp(K_k)$ be the algebra generated by the action of the Hecke operators $T_\ell$ with $\ell$ prime and $\ell \neq p$. Since $K_k \into K_{k + 1}$, we have $A_{k + 1} \onto A_k$. Let $A: = \fromlim_{k} A_k$. Then $A$ is a profinite ring embedding into $\eend(K)$: it is the shallow Hecke algebra acting on forms of level one mod~$p$. The standard pairing $A \times K \to \fp$ given by $\langle T, f \rangle \mapsto a_1(Tf)$ is nondegenerate on both sides and continuous in the profinite topology on $A$. Therefore $A$ is in continuous duality with~$K$.  
By work of Tate and Serre \cite{Tmod2, Smod3} we know that $\Delta$ is the only Hecke eigenform in $K \otimes \fpbar$.\footnote{Alternatively, one can use an observation of Serre to conclude that any Hecke eigenform in $K\otimes \algebraiclosure\F_p$ is in fact defined over $\fp$, reducing the eigenform search to a finite computation. See \cite[section 1.2 footnote]{BK}.} This implies that $A$ is a local $\fp$-algebra with maximal ideal~$\mm$ and residue field $\fp$ generated by the modified Hecke operators $T_\ell' := T_\ell - a_\ell(\dd)$, acting locally nilpotently on $M = \fp[\dd]$.  Using deformation theory, we deduce:
\begin{prop}\label{twogenprop}
There is a surjection $\fp\lb x, y \rb \onto A$ given by 
$\begin{cases} 
x \mapsto T_3,\ y \mapsto T_5 & \mbox{ if $p = 2$} \\
x \mapsto T_2,\ y \mapsto T'_7  & \mbox{ if $p = 3$.} 
\end{cases}$
\end{prop}
 
For $p = 2$, the fact that $A$ is generated by $T_3$ and $T_5$ was first proved without deformation theory in~ \cite{NS1}. For $p = 3$, Proposition \ref{twogenprop}  is stated \cite[appendix]{BK}, using deformation theory of reducible pseudocharacters, as in \cite{Bpseudodef}. See also \cite[Chapter 7]{medved} for detailed deformation theory arguments for reducible Galois pseudocharacters in level one. 

The main result of this section is the following theorem:  

\begin{thm}\label{heckeappthm}
The surjection $\fp\lb x, y \rb \onto A$ of Proposition \ref{twogenprop} is an isomorphism. 
\end{thm}

The key input will be Theorem \ref{toyngtthm}, as well as the following observation: if $T$ is any Hecke operator and $f$ is a modular form in a Hecke-invariant algebra $M$, then the sequence $\{T(f^n)\}_n$ satisfies an $M$-linear recursion. For more details on the Hecke recursion, see \cite[Chapter 6]{medved} or the forthcoming \cite{medved:heckedim}, but here we will only need some special cases for $f = \dd$ already given in \cite{NS1} and \cite[appendix]{BK}. For $p = 2$, we have (see \cite[equations (13)--(14)]{NS1})
\begin{equation}\label{eq:rec2}
\begin{aligned}
T_3(\dd^n) &= \dd \: T_3(\dd^{n-3}) + \dd^4\: T_3(\dd^{n - 4}), \quad n\geq 3\\
T_5(\dd^n) & = \dd^2 \: T_5(\dd^{n-2}) + \dd^4\: T_5(\dd^{n-4}) + \dd\: T_5(\dd^{n-5}) + \dd^6 \: T_5(\dd^{n-6}),\quad n \geq 6
\end{aligned}\end{equation}
with companion polynomials $P_3 = X^4 + \dd X + \dd^4$ and $P_5 = X^6 + \dd^2 X^4 + \dd^4 X^2 + \dd X + \dd^6$. Note that $\{T_5(\dd^n)\}_n$ also satisfies the recursion defined by $P'_5 = P_5(X^2 + \dd^2)= X^8 + \dd X^3 + \dd^3 X + \dd^8$. And for $p = 3$, the recursions satisfied by $\{T_2(\dd^n)\}$ and $\{T'_7(\dd^n)\}$ have companion polynomials
\begin{equation}\label{eq:rec3}
\begin{aligned}
P_2 &= X^3 - \dd X + \dd^3,\\
P_7 &= X^9 - \dd X^5 -  \dd^2 X^4 + (\dd^4 - \dd) X^2 + (\dd^5 + \dd^2)X  - \dd^9.
\end{aligned}\end{equation}
See Lemma 33\footnote{Lemma 33 {\it loc.\ \!cit.} gives the degree-8 recurrence satisfied by $\{T_7(\dd^n)\}_n$; the recurrence satisfied by $\{\dd^n + T_7(\dd^n)\}_n$ has an extra factor of $X - \dd$.} in \cite[appendix]{BK}. 

\begin{proof}[Proof of Theorem \ref{heckeappthm}]
 Let $T$ and $S$ be the generators of $A$ from Proposition \ref{twogenprop}. Then $T, S$ are filtered and degree-lowering recursion operators on $\fp[\dd]$, each satisfying the conditions of Theorem \ref{toyngtthm}. In other words, there exists an $\alpha < 1$ so that $N(\dd^n) : = N_T(\Delta^n) + N_S(\dd^n) \ll n^\alpha$.
 
 We now claim that the Hilbert-Samuel function of $A$ grows \emph{faster than linearly}, so that the Krull dimension of $A$ is at least $2$. Indeed, the Hilbert-Samuel function of $A$ sends a positive integer $k$ to  
\begin{align*}
\dim_{\fp} A/\mm^k = \dim_{\fp} K[\mm^k] &\geq \#\{n : \dd^n \in K,\ \mm^k \dd^n = 0\}\\ &= \#\{ n \mbox{ prime to $p$ } : N(\Delta^n) \geq k\} \gg k^{1/\alpha}, 
\end{align*}
which is certainly faster than linear, since $\frac{1}{\alpha} > 1$. Therefore it grows at least quadratically, and the Krull dimension of $A$ is at least $2$. By the Hauptidealsatz, the kernel of the surjection $\fp\lb x, y \rb \onto A$ from Proposition \ref{twogenprop} is trivial.
\end{proof}

Using the more precise bounds on $\alpha$ from Theorem \ref{specialngtthm}, we can conclude that, for $p = 2$ we have $\alpha = \max\{\log_4 2, \log_8 4\} = \frac{2}{3}$; and for $p = 3$ we have $\alpha = \max\{\log_3 2, \log_9 6\} \approx 0.815$. Compare to $\alpha = \frac{1}{2}$ obtained for $p = 2$ by Nicolas and Serre in \cite[4.1]{NS1}. Computations suggest that $\alpha = \frac{1}{2}$ also holds for $p = 3$, but we have not been able to prove this.

\section{The proof of the NGT begins}\label{proofbeginsec}

We now begin the proof of Theorem \ref{ngtthm}. 
\subsection{Overview of the proof}\label{overproofsec}
The proof proceeds as follows. 

\begin{enumerate}
\item {\bf Reduce the NGT to the case where $R$ is a finite field:} See subsection~\ref{reducetoFsec} below. 
\item {\bf Reduce to the empty-middle case:} The NGT over a finite field is implied by a special empty-middle case (Theorem \ref{specialngtthm}), where the companion polynomial has no terms of maximal total degree except for $X^d$ and $y^d$ (i.e., the highest-degree homogeneous part has an empty middle).  Note that Theorem \ref{specialngtthm} holds over any ring of characteristic $p$. See subsection \ref{specialngtsec} below for the statement of Theorem \ref{specialngtthm} and the reduction step. 
\item {\bf Prove Theorem \ref{specialngtthm}:} The main idea of the proof is as follows. Given the operator $T$ satisfying the conditions of Theorem \ref{specialngtthm}, we define a function $c_T: \N \to \N$ that grows like $n^{\alpha}$ for some $\alpha < 1$. We extend this function to polynomials in $R[y]$ via the degree. Finally, we use strong induction to prove that applying $T$ strictly lowers the $c_T$-value of any polynomial in $R[y]$. Therefore, $N_T(y^n)$ is bounded by $c_T(n) \asymp n^{\alpha}$.

The key features of this kind of proof are already present in the proof of Theorem \ref{toyngtthm} in section \ref{toyngtsec}. 
\end{enumerate}

\subsection{Reduction to the case where $R$ is a finite field}\label{reducetoFsec}

\begin{prop}\label{reducetoFprop}
If Theorem \ref{ngtthm} is true whenever $R$ is a finite field, then Theorem \ref{ngtthm} is true.
\end{prop}
\begin{proof}

First, suppose $R$ is a finite artinian local ring with maximal ideal $\mm$ and finite residue field $\F$. Let $\ell$ be the order of nilpotence of $\mm$ (that is, $\ell$ is the least positive integer so that $\mm^{\ell} = 0$).

Let $T: R[y] \to R[y]$ be the operator in the statement of the theorem, and write ${\overline T}: \F[y] \to \F[y]$ for the operator obtained by tensoring with with the quotient map $R \onto \F$. Theorem \ref{ngtthm} for $\F$ guarantees that $N_{\overline T}(y^n) \ll n^{\alpha}$ for some $\alpha < 1$. Let $$g(n): = \max_{n' \leq n} \{N_{\overline T}(y^{n'}) + 1\} \ll n^\alpha,$$ so that $g$ is nondecreasing, integer-valued, and satisfies ${\overline T}^{g(\deg f)} f=0$ for every $f$ in $\F[y]$. Lifting back to $R$, we get that $T^{g(\deg f)} f$ is in $\mm[y]$ for all $f \in R[y]$.\footnote{If $R$ is any ring and $\mathfrak a \subset R$ is an ideal, then ${\mathfrak a}[y] \subset R[y]$ is the ideal of polynomials all of whose coefficients are in ${\mathfrak a}$.} More generally, if $f$ is in $\mm^i[y]$, then $T^{g(\deg f)}$ sends $f$ to $\mm^{i + 1}[y]$. Since $\mm^{\ell} = 0$, we have $T^{\ell \, g(\deg f)} f = 0$ for every $f \in R[y]$, so that $N_T(y^n) \leq \ell\, g(n) - 1 \ll n^\alpha$.  

In the general case, $R$ is a finite product of finite artinian local rings $R_i$, and an $R$-linear operator $T: R[y] \to R[y]$ decomposes as $\sum T_i$ where $T_i: R_i[y] \to R_i[y]$ is the ($R_i$-linear) restriction of $T$ to $R_i$. From the paragraph above, we can choose $\alpha_i < 1$ so that $N_{T_i}(y^n) \ll n^{\alpha_i}$ for all $n$. Then $\alpha = \max_i \{\alpha_i\}$ works for $T$. 
\end{proof}

\subsection{Reduction to the empty-middle NGT}\label{specialngtsec}

From now on, we fix a prime $p$. Theorem $1$ over a finite field of characteristic $p$ is implied by the following special case in which the shape of the recursion satisfied by $T$ is restricted. However, note that the statement below has no finiteness restrictions on the base ring, and no restriction on the coefficient on $y^d$.

\begin{thm}[Empty-middle NGT]\label{specialngtthm}
Let $R$ be a ring of characteristic $p$, and suppose that $T$ is a degree-lowering linear operator on $R[y]$ so that the sequence $\{T(y^n)\}_n$ satisfies a linear recursion whose companion polynomial has the shape $$X^d + a y^d + (\mbox{terms of total degree} \leq d - D)$$ for some $D \geq 1$ and some constant $a \in R$.  
Let $b \geq d$ be a power of $p$, and suppose that either $b - d \leq 1$ or that $D \leq \frac{b}{2}$. Then 
 $$N_T(y^n) \ll n^\alpha \quad \mbox{for } \alpha = \frac{\log (b - D)}{\log b}.$$

\end{thm}
The case $d = b$ and $D = 1$ has already been established in Theorem \ref{toyngtthm} (in fact, the same argument extends to $d = b$ and any $D$ with $1 \leq D \leq b - 1$).

\begin{prop}[Empty-middle NGT implies NGT]\label{specialngtimpliesngtprop}\mbox{}\\
Theorem \ref{specialngtthm} implies Theorem \ref{ngtthm}.
\end{prop}

\vspace{-10pt}

\begin{proof}
By Proposition \ref{reducetoFprop}, we may assume that we are working over a finite field $\F$.  
Let $$P = X^d + a_1 X^{d - 1} + \cdots + a_d \in \F[y][X]$$ be the filtered recursion satisfied by the sequence $\{T(y^n)\}_n$ as in the setup of Theorem \ref{ngtthm}; recall that we insist that $\deg a_d = d$. We will show that $P$ divides a polynomial of the form $$X^{e} - y^e + (\mbox{terms of total degree} < e)$$ for $e = q^m(q-1),$ where $q$ is a power of $p$ and $m \geq 0$. Then the sequence $\{T(y^n)\}_n$ will also satisfy the recursion associated to a polynomial whose shape fits the requirements of Theorem~\ref{specialngtthm}. 

Let $H$ be the degree-$d$ homogeneous part of $P$, so that $P = H + (\mbox{terms of total degree} < d)$. We claim that there exists a homogeneous polynomial $S \in \F[y, X]$ so that $H\cdot S = X^e - y^e$ for some positive integer $e$ of required form. Once we find such an $S$, we know that $P\cdot S$ will have the desired shape  $X^{e} - y^e + (\mbox{terms of total degree} < e)$. 

To find $S$, we dehomogenize the problem by setting $y = 1$: let $h(X) := H(1, X) \in \F[X]$, a monic polynomial of degree $d$ and nonzero constant coefficient.
Let $\F'$ be the splitting field of $h(X)$; under our assumptions on $a_d$, all the roots of $h(X)$ are nonzero. Let $q$ be the cardinality of $\F'$. (Recall that we are assuming that $\F$, and hence its finite extension $\F'$, is a finite field.) Every nonzero element $\alpha \in \F'$, and hence every root of $h(X)$, satisfies $\alpha^{q - 1} = 1$.

Finally, let $q^m$ be a power of $q$ not less than any multiplicity of any root of $h(X)$. Since every root of $h$ satisfies the polynomial $X^{q - 1} - 1$, we know that $h(X)$ divides the polynomial $\big(X^{q - 1} - 1\big)^{q^m} = X^{q^m (q - 1) } - 1$. Set $e = q^m (q - 1)$, and let $s(X)$ be the polynomial  in $\F[X]$ satisfying $h(X) s(X) = X^e - 1$.  

Now we finally ``rehomogenize'' again: if $S \in \F[y, X]$ is the homogenization of $s(X)$, then $Q \cdot S = X^e - y^e$, so that $S$ is the homogeneous scaling factor for $P$ that we seek.  
\end{proof}

\subsection{The main induction for the proof the empty-middle NGT}\label{ngtpf}

From now on, having already fixed~$p$, we will always assume that $R$ is a ring of characteristic~$p$, not necessarily finite. 

\begin{defn} If $T: R[y] \to R[y]$ an $R$-linear operator, we will call $T$ a \emph{$(d,D)$-NRO}, for \emph{nilpotent recursion operator}, if $T$ satisfies the conditions of Theorem \ref{specialngtthm}: that is, $T$ lowers degrees, and $\{T(y^n)\}$ satisfies an $R[y]$-linear recursion with companion polynomial  $$X^d + a y^d + \mbox{(terms of total degree $\leq d - D$)}$$ for some $d \geq 1$ and some $D \geq 1$. Note that any $(d, D)$-NRO is a $(d, D')$-NRO for any $1 \leq D' \leq D$. 
\end{defn}

The proof of Proposition \ref{specialngtimpliesngtprop} shows that, if $R$ is a finite field, then any $T$ satisfying the conditions of Theorem \ref{ngtthm} is in fact a $(d, D)$-NRO for some $d$ and $D$. 

On the other hand, we make the following definition, for any triple $(b, d, D)$ with $b \geq d \geq D \geq 1$: 
\begin{defn}
A function $c: \Q_{\geq 0} \to \Q_{\geq 0}$ is a \emph{$(b, d, D)$-nilgrowth witness} if it satisfies the following properties:   
\begin{enumerate}
\item{\bf Discreteness:} $c(\N)$ is contained in a lattice of $\Q$ (that is, $\exists M \in \N$ with $M c(\N) \subset \N$).  
\item{\bf Growth:} $c(n) \asymp n^\frac{\log (b - D)}{\log b}$ as $n \to \infty.$
\item{\bf Base property:} $0 = c(0) < c(1) < \cdots < c(d-1)$ and $c(d - D) < c(d)$.
\item{\bf Step property:} For any $k \geq 0$, and any pair $(i, j) \in \{0,1,\ldots, d\}^2$ with either $(i, j) = (d, d)$ or $i - j\geq D$, and any integers $n, m$ satisfying $db^{k} \leq n < db^{k +1}$ and $jb^k \leq m$ we have 
 $$c(n) - c(n - i b^k) \geq c(m) - c(m - j b^k).$$
\end{enumerate}
\end{defn}

In this section, we prove, using strong induction, that if $b$ is a power of $p$, then the growth of the nilpotence index of a $(d, D)$-NRO is bounded by the growth of a $(b, d, D)$-nilgrowth witness:
\begin{prop}\label{nilgrowthprop}
Let $T$ be a $(d, D)$-NRO, and $b$ a power of $p$ not less than $d$. If $c$ is a $(b, d, D)$-witness, then $N_T(y^n) \leq c(n)$. 
\end{prop}

Using the Growth Property above, we obtain an immediate corollary:

\begin{cor}\label{cor:nilgrowthinduction}
Suppose that $T$ is a $(d, D)$-NRO, and let $b = p^{\lceil \log_p d \rceil}$. If there exists a $(b, d, D)$-witness, then $N_T(y^n) \ll n^\frac{\log(b - D)}{\log{b}}$. 
\end{cor}

In other words, in this section we reduce the proof of Theorem \ref{specialngtthm} to establishing the existence of a $(p^{\lceil \log_p d \rceil}, d, D)$-witness for the $(d, D)$-NRO $T$, provided that $D$ is not too big. 

\begin{proof}[Proof of Proposition \ref{nilgrowthprop}]
Given a $(b, d, D)$-witness $c$, we define a new function $\cc : R[y] \to \N \cup \{-\infty\}$ via $$\cc\big(\sum a_n y^n\big) :=  \max\{c(n): a_n \neq 0\} \qquad \mbox{and } \cc(0) := -\infty.$$

We will show that $T$ lowers the $\cc$-value of polynomials in $R[y]$: that is, that for any nonzero $f \in R[y]$, we have $\cc\big(T(f)\big) < \cc(f)$. 

It suffices to show this for $f = y^n$. 

Write $x_n$ for $T(y^n)$. We will use strong induction to show that $\cc(x_n) < c(n)$. 

The base case is all $n$ with $0 \leq n < d$. Since $\deg x_n < n$, the statement $\cc(x_n) < c(n)$ for $n < d$ is implied by the statement that $c$ is strictly increasing on $\{0, 1, \ldots, d-1\}$. This is the base property above.  

For $n > d $, let $k \geq 0$ be the integer so that $d \cdot b^{k} \leq n < d \cdot b^{k + 1}$.  Let $P(X) \in \F[y][X]$ be the companion polynomial of the given recursion satisfied by the sequence $\{x_n\}$. Let $$\mathcal I  := \{(i, j) : 0 \leq j < j + D \leq i \leq d\} \cup \{(d,d)\}.$$   By assumption, $P$ has the form  

$$P = X^d +  \sum_{(i, j) \in \mathcal I} a_{i, j} y^j X^{d - i}$$
 for some $a_{i, j} \in R$.  
By Corollary \ref{cor:ppowercharp}, the sequence $\{x_n\}$ also satisfies the order-$db^k$ recursion corresponding to $P^{b^k}$: namely, for all $n \geq d b^k$, we have

$$x_n = - \sum_{(i, j) \in \mathcal I} a_{i, j} y^{j b^k} x_{n - i b^k}.$$

We will show that, if $y^m$ appears with nonzero coefficient in one of the terms on the right-hand side above, then $c(m) < c(n)$. Since $\cc(x_n)$ is equal to one of these $c(m)$s, this will imply our claim. 
So suppose that $y^m$ appears with nonzero coefficient in the $(i, j)$-term on the right-hand side. That is, $y^m$ appears in $y^{ j b ^k} x_{n - i b^k}$ for some $(i, j) \in \mathcal I$. That means that $y^{m - j b ^k}$ appears in $x_{n - i b^k}$. Note that $i \geq D$, so that $n - i b^k < n$, and the induction assumption applies: since $y^{m - j b ^k}$ appears in $x_{n - i b^k}$, we can assume that $c(m - j b^k) < c(n - i b^k)$. 

To show that $c(m) < c(n)$, it therefore suffices to show that $$c(n) - c(m) \stackrel{?}\geq c(n - i b^k) - c(m - j b^k),$$ 
since the latter is assumed to be strictly positive. But this, slightly rearranged, is just the step property from the definition of a $(b, d, D)$-witness above. 
\end{proof}

We now aim to construct a $(b, d, D)$-witness if $b-d \leq 1$ or if $D \leq \frac{b}{2}$. This will occupy the next three sections. In section \ref{contsec} we investigate the properties of a \emph{content} function, which writes numbers in one base and reads them in another. In section \ref{sec:ineqs}, we establish some inequalities about the content of proper fractions. In section \ref{nilgrowthsec}, we use the content function to construct a $(b, d, D)$-nilgrowth witness, completing the proof of Theorem \ref{specialngtthm}.

\section{The content function and its properties}\label{contsec}

In this section we will introduce a function $c: \Q_{\geq 0} \to \Q_{\geq 0}$ that will serve as a nilgrowth witnesses in the proof of Theorem \ref{specialngtthm}. This type of function was first introduced by \jb\ in the appendix to \cite{BK}, inspired in a very loose way by the Nicolas-Serre code in \cite{NS1}. 

\subsection{Base-$b$ representation of numbers} \label{contsub}
We fix an integer $b \geq 2$ to be the \emph{base}. 

Let $\mathcal D(b) = \{0, \ldots, b-1\}$ be the alphabet of digits base $b$, and $\mathcal D(b)^\ast$ the set of words on $\mathcal D(b)$, including the empty word $\epsilon$. The number-of-letters function for a word $x \in \mathcal D(b)^\ast$ will be denoted by $\ell(x)$, for \emph{length}. 

Let $\mathcal R(b)$ be the set of all bi-infinite pointed words $$\underline x = \ldots x_2 x_1 x_0.x_{-1} x_{-2} \ldots$$ on $\mathcal D(b)$ that start with ${}^\infty 0$ (the digit $0$ repeated infinitely to the left). The set of finite words $\mathcal D(b)^\ast$ naturally embeds into $\mathcal R(b)$ via $x \mapsto ({}^\infty 0) x. (0^\infty)$, where $0^\infty$ is the digit $0$ repeated infinitely to the right. More generally, any pointed right-infinite word will be viewed as an element of $\mathcal R(b)$ by appending ${}^\infty 0$ on the left. For $\underline x \in \mathcal R(b)$, we can define the real number $\reed_b(\underline x) \in \R_{\geq 0}$ by reading it as a sequence of digits base $b$, via $\reed_b(\underline x) :=  \sum_i x_i b^i.$ Since $x_i = 0$ for $i \gg 0$, this sum converges.  The map $\reed_b$ is not injective: indeed, $\sum_{i < k} (b-1) b^i = b^k$, so that that for any finite word $w$ and digit $x \neq  b-1$ (and any radix point placement), we have $\reed_b(w\, x \,(b-1)^\infty) = \reed_b(w \, (x + 1) \,0^\infty)$. But we can choose a section of $\reed_b$ by restricting the domain: let $\mathcal R'(b) \subset \mathcal R(b)$ the subset of those that do not end with $(b-1)^\infty$. Then the reading-base-$b$ function $\reed_b: \mathcal R'(b) \to \R_{\geq 0}$ is a bijection, and the inverse map $\rite_b: \R_{\geq 0} \to \mathcal R'(b)$ takes a nonnegative real number $q$ to its \emph{normal} (that is, not ending in $(b-1)^\infty$) base-$b$ representation $\underline x = \rite_b(q)$ satisfying $\reed_b(\underline x) = q$.

The base-$b$ representation $\rite_b(q)$ is eventually periodic (that is, ends with $z^\infty$ for some finite word $z$) if and only if $q \in \Q_{\geq 0}$. For $q \in \Q_{\geq 0}$, then, we know that $$\rite_b(q) =  x.y z^\infty,$$
where $x, y, z$ are in $\mathcal D(b)$. If we insist that $x$ does not start with $0$, that first $y$ and then $z$ have minimal length among such representations, and finally that $z \neq (b-1)$, then $x$, $y$, and $z$ are defined uniquely. We will assume this minimality from now on. Note that by construction $x$ and $y$ may be empty words, but $z$ has length at least $1$. 

We define, then, three constants associated associated to $q \in \Q_{\geq 0}$: 
\begin{align*}
\ell(q) &= \ell_b(q) := \ell(x) = \max\{0, \lfloor \log_b q \rfloor + 1\},\\
s(q) &= s_b(q) := \ell(y)  = \min\{k \geq 0: \mbox{denominator of } b^k q \mbox{ is prime to } b\},\\
t(q) &= t_b(q) := \ell(z)  = \min\{k \geq 1: \mbox{denominator of } b^{s(q)} q \mbox{ divides } b^k - 1\}.
\end{align*}
In particular, we know that, for $q \in \Q_{\geq 0}$, we have 
\begin{equation}\label{eq:ratq} 
q = n + \frac{u}{b^{s(q)}} + \frac{m}{b^{s(q) }(b^{t(q)} - 1)},
\end{equation} where $n$, $u$, and $m$ are all integers with $n = \lfloor q \rfloor = \reed_b(x)$, $u = \reed_b(y)$, and $m = \reed_b(z)$. 

We will need the following very simple lemma. 
\begin{lemma}\label{integralitylemma}
Given $q \in \Q_{\geq 0}$ and a base $b$, if $q' \in q\N$, then 

\begin{enumerate}
\item $ s_b(q') \leq s_b(q),$
\item$ t_b(q')  \mbox{ divides } t_b(q).$
\end{enumerate}
\end{lemma}
The proof follows from the fact that the denominator of $q'$ is a divisor of the denominator of $q$. Alternatively, one can consider the effect of multiplication by integers on base-$b$ expansions.

\subsection{The content function}\label{contentdefsec}
Now let $b, \beta \geq 2$ be bases. Define the \emph{$(b, \beta)$-content} of $q \in \R_{\geq 0}$ (we will soon restrict to rational $q$) to be the result of reading the base-$b$ representation of $q$ in base $\beta$: 
$$c_{b, \beta}(q) := \reed_\beta\big( \rite_b(q) \big).$$

Note that $\reed_\beta$ makes sense as a function $\mathcal R(b) \to \R_{\geq 0}$: the series $\sum_{i \leq k} x_i b^i$ always converges if the $x_i$ are bounded.  

\begin{example}
\begin{enumerate}
\item Since $\rite_5(196) =1241$, we have $c_{5, 3}(196) = 1\cdot3^3 + 2\cdot3^2 + 4\cdot3 + 1 = 58.$
\item We have $\rite_7 (\frac{1}{3}) = 0.(2)^\infty$. Therefore $c_{7, 5}(n) = 2 \sum_{i \geq 1} 5^{-i} = \frac{1}{2}$. 

\item $c_{8, 3}(\frac{1}{6}) = \reed_3(0.1(25)^\infty) = \frac{1}{3} + (\frac{2}{3^2} + \frac{5}{3^3})\sum_{i \geq 0} 3^{-2i} = \frac{1}{3} + \frac{11}{27}\cdot \frac{9}{8} = \frac{19}{24}.$
\end{enumerate}
\end{example}

The following lemma, which will be used frequently, is a direct computation. 
\begin{lemma}\label{lemma:contentcompute}
For $q \in \Q_{\geq 0}$, let $s = s_b(q)$ and $t = t_b(q)$. 
Then if $q = n + \frac{u}{b^{s}} + \frac{m}{b^{s}(b^{t} - 1)}$ as in equation \eqref{eq:ratq} above, we have 
$$c_{b, \beta}(q) = c_{b, \beta}(n) + \frac{ c_{b, \beta}(u)}{\beta^s} + \frac{c_{b, \beta}(m)}{\beta^{s} (\beta^{t} - 1)}.$$
\end{lemma}

We will also use the following growth estimate: 

\begin{lemma}\label{contentgrowthlemma}
We have $c_{b, \beta}(n) \asymp n^{\log_b \beta}.$ More precisely, for $n \geq 1$, we have

$$ \beta^{-1} n^{\log_b \beta} \quad < \quad c_{b, \beta}(n) \quad < \quad \frac{\beta (b-1)}{\beta - 1}  n^{\log_b \beta}.$$

\end{lemma}

\begin{proof}
Let $\ell = \ell_b(n)$, so that for $n \geq 1$ we have $\ell = 1 + \lfloor \log_b n  \rfloor$, or $\log_b n < \ell \leq 1 + \log_b n$. Then, on one hand, the $(b, \beta)$-content of $n$ is bounded above by $\pi_\beta$ of the infinite pointed word $(b-1)^\ell.(b-1)^\infty$:
$$c_{b, \beta}(n) < \sum_{k < \ell}  (b-1) \beta^k = \frac{b-1}{\beta - 1} \beta^\ell 
\leq \frac{\beta (b-1)}{\beta - 1} \beta^{\log_b n}  
= \frac{\beta (b-1)}{\beta - 1} n^{\log_b \beta}.$$
On the other hand, the $(b, \beta)$-content of $n$ is at least $\pi_\beta$ of the pointed word $1(0^{\ell - 1}).0^\infty$: 
$$c_{b, \beta}(n) \geq \beta^{\ell - 1} > \beta^{-1} n^{\log_b \beta}.$$
\end{proof}
In particular, if $\beta < b$, then $c_{b, \beta}$ grows slower than linearly in $n$. 
\begin{rk}
\begin{enumerate}
\item 
If $\beta < b$, then the $(b, \beta)$-content function is never monotonically increasing, even on integers. Indeed, $c_{b, \beta}(b^k-1) = \frac{b-1}{\beta - 1}(\beta^k - 1)$, which is greater than $c_{b, \beta}(b^k) = \beta^k$ as soon as $\beta^k > \frac{\beta - 1}{b - 2}$. 
\item 
The sequence $\{ c_{b, \beta}(n) \}_{n \in \N}$ is $b$-regular in the sense of Allouche and Shallit \cite{alloucheshallit}. They define a sequence $\{s_n\} \in \Q^\N$ to be $b$-regular if there exists a $\Q$-vector space $V$, endomorphisms $M_0, \ldots, M_{b-1}$ of $V$, a vector $v \in V$, and a functional $\lambda: V \to \Q$ so that $s_{\pi_b(n_\ell \ldots n_0)} = \lambda M_{n_\ell} \cdots M_{n_0} v$. For $\{c_{b, \beta}(n)\}$, we can take $V = \Q^2$, $M_i = \begin{psmallmatrix} 1 & i \\ 0 & \beta \end{psmallmatrix}$, $v = \begin{psmallmatrix} 0 \\ 1 \end{psmallmatrix}$ and $\lambda = \begin{psmallmatrix} 1 & 0 \end{psmallmatrix}$. Indeed, $b$-regularity appears to be lurking in many places in this theory\footnote{For example, the Nicolas-Serre code sequence $h(n)$ defined in \cite[\S 4.1]{NS1} is $2$-regular, as are its constituent parts $n_3(n)$ and $n_5(n)$.}, but we have not yet been able to make use of it.  
\end{enumerate}
\end{rk}

Finally, we record a trivial property of $c_{b, \beta}$: 
\begin{lemma}\label{lemma:scalingcontent}
For $n \in \R_{\geq 0}$, and any $k \in \Z$, we have $c_{b, \beta}(b^k n) = \beta^k c_{b, \beta}(n)$. 
\end{lemma}

\subsection{Content and the carry-digit word}\label{sec:carryword}
Even though the content function is not monotonic, it exhibits a certain amount of regularity under addition. We quantify this regularity in this section.
 
For $m, n \in \R_{\geq 0}$, we define $r_b(m, n) \in \mathcal R(2)$ to be the word of carry digits when the sum $s = m + n$ is computed in base $b$. More precisely, let $\rite_b(m) = \underline m$, $\rite_b(n) = \underline n$, and $\rite_b(s) = \underline s$, and let $r_b(m, n) : = \underline r$ satisfying 
\begin{equation} \label{eq:carrydigit}
m_i + n_i + r_{i -1} = s_i + b r_i \quad \mbox{ for all $i$ in $\Z$. }
\end{equation}
Since $m_i$, $n_i$, and $s_i$ are all $0$ for $i > \ell_b(s)$, and since $r_i \in \{0, 1\}$, the set of equations above defines $r_i$ uniquely, inductively down from $i = \ell_b(s)$. 

\begin{example}
\begin{enumerate}
\item We compute $r_3(77, 11)$. We have $\rite_3(77) = 2212$ and $\rite_3(11) = 102$. Since the expansions are finite, the carry digits are computed simply in performing the addition algorithm base $3$:  
$$\begin{array}[t]{r}
\overset{1}{\phantom{1}}\overset{1}{2}\overset{0}2 \overset{1}1 \overset{}2\\
+ 102\\ 
\hline
10021
\end{array}$$
From this we read off the carry-digit word: $r_3(77, 5) = 1101.$ Note the shift one space to the right in our indexing from how one conventionally keeps track of carry digits in long addition. 

\item We compute $r_5(\frac{53}{60}, \frac{23}{100})$. We have 
$\rite_5\left(\textstyle\frac{53}{60}\right) = 0.4(20)^\infty$ and $\rite_5\left(\textstyle \frac{23}{100}\right) = 0.10(3)^\infty.$ 
Their sum is $\frac{167}{150} = \reed_5(1.02(40)^\infty)$. Comparing the base-5 expansions of the two addends with the expansion of the sum allows us to compute the carry digits left to right. 
$$\begin{array}[t]{r}
\overset{1}{0}.\overset{0}4 \overset{0}2 \overset{1}0\overset{0}2\overset{1}0\overset{0}2\overset{1}0\overset{0}2\overset{\ldots}{\phantom{9}}\hspace{-7pt}\ldots\\
+ 0.10333333\ldots \\ 
\hline
1.02404040\ldots
\end{array}$$
Therefore $r_5(\frac{53}{60}, \frac{23}{100}) = 0.10(01)^\infty$. In this case, we will get the same infinite carry-digit word if we take the ``limit" of the finite carry-digit words obtained by truncating the expansions of the two addends. 

\item We insist that $r_{10}(\frac{1}{3}, \frac{2}{3}) = 0.1^\infty$, even though any finite truncation of the decimal expansions of $\frac{1}{3}$ and $\frac{2}{3}$ would yield no carry digits in the sum. In full generality (i.e., if the addends are not in $\Z[\frac{1}{b}]$ but the sum is), one needs to know the expansion of the sum before one can compute the carry-digit word. 
\end{enumerate}

\end{example}

The carry digit word exactly keeps track of the difference between values of $c_{b, \beta}$: 

\begin{lemma} \label{lemma:carrydigits}For $m, n$ be in $\R_{\geq 0}$, we have 
$$c_{b, \beta}(m) + c_{b, \beta}(n) = c_{b, \beta}(m +  n) + (b - \beta) \, \reed_\beta\big(r_b(m, n)\big).$$
\end{lemma}

\begin{proof}
Let $s = m + n$, and let $\underline m$, $\underline n$, $\underline s$ be the corresponding base-$b$ expansions and $\underline r$ the carry-digit word. Scaling equation \eqref{eq:carrydigit} by $\beta^i$ and summing up over all $i$ gives us 
$$\sum m_i \beta^i + \sum n_i \beta^i + \sum r_{i-1} \beta^i = \sum s_i \beta^i + b \sum r_i \beta^i$$
or, equivalently,
$$c_{b, \beta}(m) + c_{b, \beta} (n) + \beta \pi_\beta(\underline r) = c_{b,\beta}(m + n) + b  \pi_\beta(\underline r).$$
\end{proof}

We will typically use Lemma \ref{lemma:carrydigits} when comparing $c_{b, \beta}(m)$ and $c_{b,\beta}(n)$ by analyzing $c_{b, \beta}(m-n)$ and $r_b(m-n, n)$. 

\begin{example}
\begin{enumerate}
\item Consider $(b, \beta) = (3, 2)$ and the $77 + 11 = 88$ example from above. We have $c_{3, 2}(77) = \reed_2(2212) = 28$, and $c_{3, 2}(88) = \reed_2(10021) = 21$. The difference is accounted for by $c_{3, 2}(11) = \reed_2(102) = 6$ and the carry digit word evaluation in base~$2$. Namely, we have $\reed_2 \big(r_3(77, 11)\big) = \reed_2(1101) = 13$, and then $28 + 6 = 21 + 13$, as in Lemma \ref{lemma:carrydigits}. 
\item Let $(b,\beta) = (5, 3)$, and consider the $\frac{53}{60}+\frac{23}{100}$ example from above. Using Lemma \ref{lemma:contentcompute}, we find that 
\begin{align*}
c_{5, 3}(\textstyle\frac{53}{60}) &= \reed_3(0.4(20)^\infty) &= \frac{\reed_3(4)}{3} + \frac{\reed_3(20)}{3(3^2 - 1)} &= \frac{19}{12}\\
c_{5, 3}(\textstyle\frac{23}{100}) &= \reed_3(0.10(3)^\infty) &= \frac{\reed_3(10)}{3^2} + \frac{\reed_3(3)}{3^2(3-1)} &= \frac{1}{2}\\
c_{5, 3}(\textstyle\frac{167}{150}) &= \reed_3(1.02(40)^\infty) &= \reed_3(1) + \frac{\reed_3(02)}{3^2} + \frac{\reed_3(40)}{3^2(3^2 - 1)} &= \frac{25}{18}\\
\reed_3\big(r_5(\textstyle\frac{53}{60}, \frac{23}{100})\big) &= \reed_3\big(0.10(01)^\infty\big) &= \frac{\reed_3(10)}{3^2} + \frac{\reed_3(01)}{3^2(3^2 - 1)} &= \frac{25}{72}.
\end{align*}
As expected from Lemma \ref{lemma:carrydigits}, we have $\frac{19}{12} + \frac{1}{2} = \frac{25}{12} = \frac{25}{18} + 2\cdot \frac{25}{72}$. 

\item Finally, let $b = 10$ and return to the addition equation $\frac{1}{3} + \frac{2}{3} = 1$. For $a \in \mathcal D(9)$, we have $\reed_\beta(0.a^\infty) =  \frac{a}{\beta - 1}$. Therefore the two sides of the Lemma \ref{lemma:carrydigits} equation agree:
\begin{align*} 
& \mbox{(LHS)} \quad c_{10, \beta}(\textstyle \frac{1}{3}) + c_{10, \beta}(\frac{2}{3}) &= & \reed_\beta(0.3^\infty) + \reed_\beta(0.6^\infty) & =  \frac{9}{\beta - 1}\\
& \mbox{(RHS)}\quad c_{10, \beta}(1) + (10 - \beta)\reed_\beta\big(r_{10}(\textstyle \frac{1}{3}, \frac{2}{3})\big)   &=& 
1 + (10 - \beta) \reed_\beta(0.1^\infty) & = \displaystyle \frac{9}{\beta - 1}.
\end{align*}
\end{enumerate}

\end{example}

\section{Content of some proper fractions}\label{sec:ineqs}
In this section, we continue notation from the previous section, but further assume that $d > 1$.  We prove some inequalities about $c_{b, \beta}(\frac{1}{d})$ and $c_{b, \beta}(\frac{D}{b})$ that we will use in section \ref{nilgrowthsec} to produce a $(b, d, D)$-nilgrowth witness. 

\subsection{Unit fractions base $b$}
To motivate the discussion, we note that  $$\frac{1}{d} = \frac{\frac{1}{b}}{1 - \frac{b-d}{b}} = \sum_{k \geq 1} (b-d)^{k-1} b^{-k}.$$
Therefore, the base-$b$ expansion of $\frac{1}{d}$ ``wants" to be $0.1\, (b-d) \,(b-d)^2 \,(b-d)^3\ldots$. Of course, unless $b - d \leq 1$, this is not possible: $(b-d)^k$ is not a digit base $b$ for $k$ large enough. However, letting $\rite_b(\frac{1}{d}) = 0.a_1 a_2 a_3 \ldots$, we can say the following: 

\begin{lemma}\label{lemma:aiform}
For $k \geq 1$, we have $a_i = (b-d)^{i-1}$ for $i = 1, \ldots, k$ if and only if $(b-d)^k < d$. 
\end{lemma}

\begin{proof}

We will establish this claim by induction on $k$. The $a_k$ can be defined recursively via 
\begin{equation}\label{eq:ak} 
a_k =  \left\lfloor \frac{b^k}{d}- \sum_{i = 1}^{k-1} a_i b^{k-i} \right\rfloor=  \left\lfloor \frac{b^k}{d}\right\rfloor - \sum_{i = 1}^{k-1} a_i b^{k-i}.
\end{equation}

For $k = 1$, we have $a_1 = \lfloor \frac{b}{d} \rfloor \geq 1$. Therefore $a_1 = 1$ if and only if  $\frac{b}{d} < 2$, which is equivalent to $(b-d)^1 < d$. So the claim for $k = 1$ is true.  

Now suppose we already know that $a_i = (b-d)^{i-1}$ for $i < k$. Then $a_k = (b-d)^{k-1}$ if and only if 
\begin{align*}
1& > \left(\frac{b^k}{d} - \sum_{i = 1}^{k-1} a_i b^{k-i}\right) - (b-d)^{k-1}\\ &= \frac{b^k}{d} - \big(b^{k-1} + (b-d) b^{k-2} + \cdots + (b-d)^{k-2} b + (b-d)^{k-1}\big) \\
&= \frac{b^k}{d} - \frac{b^k - (b-d)^k}{b - (b-d)} = \frac{(b-d)^k}{d}, 
\end{align*}
as desired.
\end{proof}

The same argument also implies the immediate
\begin{cor}\label{cor:aiform} If  $a_i = (b-d)^{i-1}$ for $i<k$, then $a_k = (b - d)^{k-1} +  \left\lfloor \frac{(b-d)^{k}}{d}  \right\rfloor.$
\end{cor}

We can now delineate what $\frac{1}{d}$ must look like base $b$.

\begin{lemma}\label{1dstructurelemma}
For $d \leq b$, the base-$b$ expansion of $\frac{1}{d}$ falls into one of five mutually exclusive cases. 
\begin{enumerate}
\item\label{0.2} $\rite_b(\frac{1}{d}) = 0.2^+ \ldots$ (in other words, $a_1 \geq 2$) iff $d \leq \frac{b}{2}$. 
\item\label{0.13} $\rite_b(\frac{1}{d}) = 0.13^+\ldots$ (i.e., $a_1 = 1$ and $a_2 \geq 3$) iff $\frac{b}{2} < d \leq \frac{b^2}{b+3} = b - 3 + \frac{9}{b + 3}$. 
\item\label{0.124} $\rite_b(\frac{1}{d}) = 0.124^+ \ldots$ (i.e., $a_1 = 1$, $a_2 = 2$, and $a_3 \geq 4$) iff $b > 6$ and $d = b - 2$.  
\item $\rite_b(\frac{1}{d}) = 0.1^\infty$ iff $d = b - 1$. 
\item $\rite_b(\frac{1}{d}) = 0.10^\infty$ iff $d= b$. 
\end{enumerate}
\end{lemma}

\begin{proof}
If $b -d = 0$ or $b - d = 1$, then $a_i = (b-d)^{i - 1}$ for all $i$ (Lemma \ref{lemma:aiform}). Assume $b - d \geq 2$. If $d \leq \frac{b}{2}$, then $\frac{1}{d} \geq \frac{2}{b}$, so that $a_1 \geq 2$, as claimed. Otherwise, we must have $(b-d)^1 < d$, so that $a_1 = 1$ and $a_2 \geq b - d$. This means that $a_2 \geq 3$ unless \emph{both} $b-d = 2$ and $(b-d)^2 < d$, in which case we have $d > 4$ (and hence $b > 6$), and $a_2 \geq (b-d)^2 = 4$.  
\end{proof}

\subsection{The carry-digit word for a proper fraction base $b$}\label{sec:Dcarryword}
Now we additionally fix a $D$ with $1 \leq D < d$ and investigate $\rite_b(\frac{D}{d})=: 0.e_1 e_2 e_3 \ldots$. The $e_k$ satisfy the same type of recursion as the $a_k$, namely, 
$$e_k =  \left\lfloor \frac{b^k D}{d} \right\rfloor- \sum_{i = 1}^{k-1} e_i b^{k-i}. $$ In particular, $e_1 = \left \lfloor \frac{Db}{d} \right\rfloor$, and the following generalization of Lemma \ref{lemma:aiform} and Corollary \ref{cor:aiform} holds: 
\begin{lemma} \label{lemma:eiform}
For $k \geq 0$, we have $e_i = D(b-d)^{i-1}$ for $i = 1, \ldots, k$ if and only if $D(b-d)^k < d$. Moreover, if $D(b-d)^k < d$, then $e_{k + 1} = D(b-d)^k + \left\lfloor \frac{D(b-d)^{k+1}}{d} \right\rfloor$. 
\end{lemma}

To understand the relationship between the $a_k$ and the $e_k$ we define the carry-digit word $\underline{r} = 0.r_1 r_2 r_3 \ldots$ for the addition problem $\underbrace{\textstyle \frac{1}{d} + \cdots + \frac{1}{d}}_{D\ {\rm times}} = \frac{D}{d}$. (See section~\ref{sec:carryword} for definitions.) In other words, set $\underline{r}^{(i)} := 0.r^{(i)} _1 r_2^{(i)}  r_3^{(i)} \ldots := r_b(\frac{i}{d}, \frac{1}{d})$ for $1 \leq i \leq D-1$, and then set $r_j := \sum_{i = 1}^{d-1} r_j^{(i)} $. 

Then $0 \leq r_j \leq D-1$ for every $i$; since $\frac{D}{d} < 1$, we have $r_1 = 0$. Putting together equations \eqref{eq:carrydigit} applied to $\frac{i}{d} + \frac{1}{d}$ for $1 \leq i < D$ gives us the precise relationship between the $a_k$, $e_k$, and $r_k$: 
\begin{equation}\label{eq:Dmult}
e_k = D a_k + r_{k + 1} - b r_k. 
\end{equation}

In fact, we have a closed formula for $r_k$: 

\begin{lemma}\label{r3lemma}
For $k \geq 1$ we have $$r_{k}  \quad = \quad \left\lfloor \frac{Db^{k-1}}{d} \right\rfloor - D\left\lfloor \frac{b^{k-1}}{d} \right\rfloor \quad = \quad \left\lfloor \frac{D(b-d)^{k-1}}{d} \right\rfloor - D\left\lfloor \frac{(b-d)^{k-1}}{d} \right\rfloor.$$
\end{lemma}

\begin{proof}
The formula is true for $k = 1$ (in our case, all quantities are $0$). For $k \geq 1$, we will establish the formula for $k + 1$, starting with equation \ref{eq:Dmult}:
\begin{align*}
r_{k + 1} =  e_k - Da_k + b r_k
&= 
\left\lfloor \frac{D b^k}{d}\right\rfloor - \sum_{i = 1}^{k-1} e_i b^{k-i}  - 
D\left\lfloor \frac{b^k}{d} \right\rfloor  + D\sum_{i = 1}^{k-1} a_i b^{k-i} + b r_k\\
&= \left\lfloor \frac{D b^k }{d}\right\rfloor - D\left\lfloor \frac{b^k}{d} \right\rfloor + 
\sum_{i = 1}^{k-1} b^{k-i} (b r_i - r_{i + 1}) + br_k
= \left\lfloor \frac{D b^k }{d}\right\rfloor - D\left\lfloor \frac{b^k}{d} \right\rfloor
\end{align*}
Here we used equation \eqref{eq:Dmult} in the form $-e_i + Da_i = b r_i - r_{i + 1}$ for $1 \leq i < k$ to pass from the first line to the second, and cancelation of a telescoping sum to pass from the second to the third. 
Finally, we note that 
$$\left\lfloor \frac{D b^k }{d}\right\rfloor - D\left\lfloor \frac{b^k}{d} \right\rfloor = \left\lfloor \frac{D (b-d)^k }{d}\right\rfloor - D\left\lfloor \frac{(b-d)^k}{d} \right\rfloor$$
because the intervening terms $\frac{D \sum_{i = 1}^k {k \choose i} b^{k-i} (-d)^i}{d}$ are integers, and hence can pass through the greatest-integer function to cancel.  
\end{proof}

\begin{cor}\label{r2cor}\leavevmode
\begin{enumerate}
\item If $(b-d)^k < d$ for some $k \geq 0$, then for every $i \leq k + 1$, we have $r_i = \left\lfloor \frac{D(b-d)^{i-1}}{d} \right\rfloor$.
\item If $D(b-d)^k < d$ for some $k \geq 0$, then for every $i \leq k + 1$, we have $r_i = 0$. 
\end{enumerate}
\end{cor}

\subsection{{$(b, \beta)$}-Content of proper fractions}\label{cdefsec}

As before, we have a triple $(b, d, D)$ with $1 \leq D \leq d \leq b$ subject to the conditions that $b \geq 2$ and $\beta: = b - D \geq 2$. Recall the content function $c_{b, \beta}$ from section \ref{contentdefsec}, and let $c_{b, \beta}^d: \Q_{\geq 0} \to \Q_{\geq 0}$ be the function defined via $$c_{b, \beta}^d(n) : = c_{b, \beta}(\textstyle \frac{n}{d}).$$
Whenever the triple $(b, d, D)$ is understood, we write $c = c_{b, \beta}^d$, and let $\underline r = 0.r_1 r_2 \ldots$ be the carry-digit word for $\frac{1}{d} + \cdots + \frac{1}{d} = \frac{D}{d}$ as in section~\ref{sec:Dcarryword}.  Lemma \ref{lemma:carrydigits} implies that 
\begin{equation} \label{eq:cDc1}
c(D) = Dc(1) - D\reed_\beta(\underline r).
\end{equation}

In this section, we will establish some lower bounds on $c(1)$ and $c(D)$.  First, we dispatch the cases $d = b$ and $d = b-1$, which yield easy explicit formulas.

\begin{lemma} \label{easycontlemma}
Suppose that $d = b$ or $d = b - 1$, and $0 \leq i < d$. Then $c(i) = \frac{i}{d - D}$.
\end{lemma}

\begin{proof}
Computation: see Lemma \ref{1dstructurelemma}. Note that $D \leq \min\{d, b - 2\}$ precludes the possibility that $d = D$. 
\end{proof}

\begin{prop}\label{cont1lemma}
If $d \leq b - 2$ and $b > 6$, then 
$c(1) \geq \frac{\beta + 1}{\beta (\beta - 1)}.$ 
\end{prop}

\begin{rk}
It is a simple exercise to check that the only exceptions for $b \leq 6$ are in fact $(b, d, D) = (4, 2, 2)$, $(5, 3, 3)$, or $(6, 4, 4)$ by exhausting all cases. 
\end{rk} 

\begin{proof} 
We go through the first three cases of Lemma \ref{1dstructurelemma}.
Note that $\beta = 2$ implies that all of the inequalities in $2 \leq b - d \leq b - D = \beta$ are equalities, so that $d = D$ and $d = b - 2$. In particular, if $\beta = 2$ and $b > 6$, then we must be in the third case. 

\begin{enumerate}
\item If $\rite_b(\frac{1}{d})$ starts with $0.2^+$, then $c(1) \geq \reed_\beta(0.2) = \frac{2}{\beta}$. Since $\beta \geq 3$, we have $2 \geq \frac{\beta + 1}{\beta - 1}$, so that $c(1) \geq \frac{\beta + 1}{\beta(\beta - 1)}$, as desired. 

\item If $\rite_b(\frac{1}{d})$ starts with $0.13^+$, then we have $c(1) \geq \rite_\beta(0.13) = \frac{\beta + 3}{\beta^2}$. This last is no less than $\frac{\beta + 1}{\beta(\beta - 1)}$ if and only if $\beta^2 + 2 \beta - 3 = (\beta + 3)(\beta - 1) \stackrel{?}\geq (\beta + 1)\beta = \beta^2 + \beta$. Therefore $\beta \geq 3$ again implies $c(1) \geq \frac{\beta + 1}{\beta(\beta - 1)}$, as desired. 

\item If $\rite_b(\frac{1}{d})$ starts with $0.124^+$, then $$c(1) \geq \reed_\beta\big(0.124\big) = \frac{\beta^2 + 2\beta + 4}{\beta^3} = \frac{\beta + 1}{\beta(\beta -1)} + \frac{2\beta - 4}{\beta^4 - \beta^3}.$$ Since $\beta \geq 2$, our claim is established. 
\end{enumerate}
\end{proof}

\begin{cor}\label{cor:cont1}
If $d \leq b - 2$ and $D \leq \frac{b}{2}$, then $c(1) \geq \frac{D}{\beta(\beta - 1)}$. 
\end{cor}

\begin{proof} 
For $D \leq \frac{b}{2}$, we have $\beta = b - D \geq b - \frac{b}{2} = \frac{b}{2} \geq D$. Therefore Proposition \ref{cont1lemma} establishes the desired inequality, the exceptional cases $(b, d, D) = (4, 2, 2)$, $(5, 3, 3)$, or $(6, 4, 4)$ being easy to check explicitly. 
\end{proof}

In the next proposition we will show that $c(D)$ is not too small, provided that $d$ is not too big relative to $b$, or, failing that, that $D$ is not too big relative to $b$ and $d$. 

\begin{prop}\label{contDlemma}
Suppose $D < d \leq b - 2$ and \emph{at least one} of the following conditions is satisfied:
\begin{enumerate}
\item\label{dcond} $d \leq \frac{b}{2};$ 
\item\label{Dcond} $D < d\,(1 - \frac{1}{b-d}).$
\end{enumerate}
Then $c(D) \geq \frac{D(\beta + 1)}{\beta (\beta - 1)}.$
\end{prop}

\begin{rk}
Computationally, it appears that the optimal statement is as follows. If $D < d \leq b - 2$, then $c(D) \geq   \frac{D(\beta + 1)}{\beta (\beta - 1)}$ if and only if at least one of the following is true: ($1^\prime$) $(b-d)^2 > b - 1$, or (2) $D < d\,(1 - \frac{1}{b-d})$. Note that Condition \eqref{dcond} above implies condition ($1^\prime$) here. Here we only prove Proposition \ref{contDlemma} as stated. 
\end{rk}

Before proving Proposition \ref{contDlemma}, some preparatory lemmas.
\begin{lemma}\label{lemma:r2ineq}
Under the assumption $d > \frac{b}{2}$, condition \eqref{Dcond} from Proposition \ref{contDlemma} is equivalent to the inequality $r_2 < b - d - 1.$
\end{lemma}

\begin{proof}
Apply Corollary \ref{r2cor} for $k = 0$ to deduce that $r_2 = \left\lfloor \frac{D(b-d)}{d} \right\rfloor$. Then $r_2 < b - d - 1$ if and only if $\frac{D(b-d)}{d} < b - d - 1$ if and only if $D < \frac{d(b - d - 1)}{b-d}$, which is condition \eqref{Dcond}, as desired. 
\end{proof}

As before, let $\rite_b(\frac{1}{d}) = 0.a_1 a_2 \ldots$,  $\rite_b(\frac{D}{d}) = 0.e_1 e_2 \ldots$, and let $\underline{r} = 0.r_1 r_2 \ldots$ be the carry digits as in section \ref{sec:Dcarryword}. Using equation \eqref{eq:Dmult} and the partial-sum cutoffs $c(D) \geq \reed_\beta(0.e_1 \ldots e_k) = \sum_{i = 1}^k e_i \beta^{-i}$ we get the following partial-sum versions of equation \eqref{eq:cDc1}:

\begin{lemma}\label{poscutofflemma}
For any $k \geq 1$, the quantity $c(D)$ satisfies the following inequality:
$$c(D) \geq \frac{D \sum_{i = 1}^k (a_i - r_i)\beta^{k-i} + r_{k + 1}}{\beta^{k}}.$$
\end{lemma}

\begin{cor}\label{lemma:a2r23}
Any of the following conditions are sufficient to guarantee $c(D) \geq \frac{D (\beta + 1)}{\beta(\beta - 1)}$.
\begin{enumerate}
\item\label{pa12} $\beta \geq 3$ and $a_1 \geq 2;$
\item\label{pr2} $r_2 \geq \frac{2D}{\beta - 1};$ 
\item\label{pa2r2} $\beta \geq 3$ and $a_2- r_ 2 \geq 3;$ 
\item\label{pr3} $a_2 - r_2 = 2$ and $r_3 \geq \frac{2D}{\beta - 1};$
\item\label{pa3r3}$\beta \geq 3$ and $a_2 - r_2 = 2$ and $a_3 - r_3 \geq 3.$
\end{enumerate}
\end{cor}

\begin{proof}[Proof (of Corollary \ref{lemma:a2r23})]
We use Lemma \ref{poscutofflemma} for each specified $k$. Recall that $r_1 = 0$. 
\begin{enumerate} 
\item $k = 1$, use estimate $r_2 \geq 0$. We have  
$c(D) \geq \frac{2D}{\beta} \geq \frac{D (\beta + 1)}{\beta(\beta - 1)},$
since $2 \geq \frac{\beta + 1}{\beta - 1}$ for $\beta \geq 3$. 
\item $k = 1$, use estimate $a_1 \geq 1$: 
$$c(D) \geq \frac{D + r_2}{\beta} \geq \frac{D + \frac{2D}{\beta - 1}}{\beta} = \frac{D (\beta + 1)}{\beta(\beta - 1)}.$$
\item $k = 2$, use estimate $a_1 \geq 1$ and $r_3 \geq 0$: 
\begin{align*}
c(D) &\geq \frac{\beta(D + r_2) + (D a_2 - b r_2)}{\beta^2} = \frac{D(\beta + a_2 -r_2)}{\beta^2}.
\end{align*} This last being greater than $\frac{D(\beta + 1)}{\beta(\beta - 1)}$ is equivalent to $(\beta + a_2 - r_2)(\beta - 1) \geq \beta(\beta + 1)$, or $a_2 - r_2 \geq \frac{2\beta}{\beta - 1}$. For $\beta \geq 3$, this is guaranteed by $a_2 - r_2 \geq 3$.
\item $k = 2$, use estimate $a_1 \geq 1$: 
$$\frac{c(D)}{D} \geq \frac{\beta + (a_2 - r_2)  + \frac{r_3}{D}}{\beta^2}
\geq \frac{\beta + 2 + \frac{2}{\beta - 1}}{\beta^2}
= \frac{\beta + 1}{\beta (\beta - 1)}.$$
\item $k = 3$, use estimate $a_1 \geq 1$ and $r_4 \geq 0$: 
$$\frac{c(D)}{D} \geq \frac{\beta^2 + (a_2 - r_2)\beta + (a_3 - r_3)}{\beta^3}
\geq \frac{\beta^2 + 2\beta + 3}{\beta^3}
= \frac{\beta + 1}{\beta (\beta - 1)} + \frac{\beta - 3}{\beta^3(\beta - 1)}.$$

\end{enumerate}

\end{proof}

\begin{proof}[Proof of Proposition \ref{contDlemma}]
 Note that the assumptions $D < d \leq b - 2$ guarantee that $\beta \geq 3$. 

If condition \eqref{dcond} holds, then $a_1 \geq 2$ (Lemma \ref{1dstructurelemma}\eqref{0.2}) so that Lemma \ref{lemma:a2r23}\eqref{pa12} gives what we want. 
If condition \eqref{dcond} fails, but condition \eqref{Dcond} holds, then by Lemma \ref{lemma:r2ineq}, we have $r_2 \leq b - d - 2$. Moreover, from Lemma \ref{lemma:eiform} for $D = 1$ and $k = 1$, we know that $a_2 \geq b - d$. If either inequality is strict, Lemma \ref{lemma:a2r23}\eqref{pa2r2} gives us the desired inequality. Therefore it remains to consider the case $r_2 = b - d - 2$ (so that $\frac{d(b - d - 2)}{b-d} \leq D < \frac{d(b - d - 1)}{b-d}$) and $a_2 = b - d$ (so that $(b - d)^2 < d$\footnote{Incidentally this implies the failure of condition ($1'$), which should conjecturally replace condition \eqref{dcond} as noted in the remark after the statement of Proposition~\ref{contDlemma}.}.) 

We now estimate the third digits. By Corollary \ref{r2cor} and Lemma \ref{lemma:eiform}, we have $r_3 = \left \lfloor \frac{D (b-d)^2}{d} \right \rfloor$ and $a_3 \geq (b-d)^2$. Condition \eqref{Dcond} implies that $r_3 \leq \frac{D (b-d)^2}{d} < (b-d)(b-d - 1)$, so that $$a_3 - r_3 > (b-d)^2 - (b-d)(b-d-1) = b-d \geq 3,$$
so that the desired inequality holds by Lemma \ref{lemma:a2r23}\eqref{pa3r3}.

\end{proof}

\section{The nilgrowth witness}\label{nilgrowthsec}
Let $b \geq d \geq D \geq 1$ be integers subject to the conditions $b \geq 2$ and $\beta : = b - D \geq 2$, as usual.  
In this section we exhibit a $(b,d, D)$-nilgrowth witness function and complete the proof of Theorem~\ref{specialngtthm}.

Recall from subsection \ref{cdefsec} that 
$c_{b, \beta}^d: \Q_{\geq 0} \to \Q_{\geq 0}$ is the function defined via $$c_{b, \beta}^d(n) : = c_{b, \beta}(\textstyle \frac{n}{d}).$$
Also define the integer constant $M_{b, \beta}^d := \beta^{s_b(1/d)}\big(\beta^{t_b(1/d)} - 1\big)$.  Here $c_{b, \beta}$ is the $(b, \beta)$-content function, first defined in section \ref{contentdefsec}, and $s_b$ and $t_b$ count the number of digits after the decimal point of the pre-period and the period, respectively, of base-$b$ expansions; see definition before equation \eqref{eq:ratq}.

The following theorem, combined with Corollary \ref{cor:nilgrowthinduction}, will prove Theorem \ref{specialngtthm}, completing, in turn, the proof of Theorem \ref{ngtthm}.  

\begin{thm} \label{witnessthm}
If $b - d \leq 1$, or if $D \leq \frac{b}{2}$, then the function $c_{b, b - D}^d$ is a $(b, d, D)$-nilgrowth witness. 

\end{thm}

We begin the proof of Theorem \ref{witnessthm}. Recall from subsection \ref{ngtpf} that a $(b, d, D)$-nilgrowth witness must satisfy four properties: discreteness, growth, base, and step. We establish the first two immediately. 

\begin{lemma}[Discreteness property]
For any $n \in \N$, we have $M_{b, \beta}^d\, c_{b, \beta}^d(n) \in \N$. 
\end{lemma}

\begin{proof}
It suffices to see that $ \beta^{s_b(1/d)}\big(\beta^{t_b(1/d)} - 1\big) c_{b, \beta}(n)$ is an integer for $n \in \frac{1}{d} \N$. For $n = \frac{1}{d}$ this follows from Lemma \ref{lemma:contentcompute}, and for general $n \in \frac{1}{d}\N$ from Lemmas \ref{lemma:contentcompute} and \ref{integralitylemma}.
\end{proof} 

\begin{lemma}[Growth property]
We have $c_{b, \beta}^d(n) \asymp n^{\log_b \beta}$. 
\end{lemma}

\begin{proof}
Lemma \ref{contentgrowthlemma}. 
\end{proof}

It remains to establish the base property and the step property.

 For $m, n \in \Q_{\geq 0}$ with $m \geq n$, set $$R(m, n) : = R_{b, \beta}^d(m, n) := D \reed_\beta r_b(\textstyle \frac{m-n}{d}, \frac{n}{d}).$$
Here $r_b$ is the carry-digit word, as in section \ref{sec:carryword}.  We then have, for $m, n$ as above 
\begin{equation}\label{eq:carryr}
c(n) - c(n - m) = c(m) - R(n,m).
\end{equation}
This is just a restatement of Lemma \ref{lemma:carrydigits}, in the form in which we will use it below. 

We now use the technical results of section \ref{sec:ineqs} to prove that our candidate nilgrowth witness satisfies the base property and the step property. 

\begin{lemma}[Base property]\label{lemma:base}\leavevmode
Suppose that $d = b$ or $d = b - 1$ or $D \leq \frac{b}{2}$. Then the following inequalities hold: 
\begin{enumerate}
\item \label{cDd} $c(d - D) \leq c(d)$;
\item \label{baseincreasing} $0 = c(0) < c(1) < \ldots < c(d -1)$.
\end{enumerate}
\end{lemma}

\begin{rk} Part \eqref{cDd} is in fact true without the assumption $D \leq \frac{b}{2}$, but we do not need this greater generality. Part \eqref{baseincreasing} above is not generally true if $D > \frac{b}{2}$. For example, for $(b, d, D) = (7, 5, 5)$, we have $c(2) = c(3) = 3$; and for $(b, d, D) = (11, 9, 7)$ we have $c(4) = \frac{334}{195} > \frac{316}{195} = c(5)$. The condition delineated here is certainly not optimal, however. 

\end{rk}

\begin{proof}\leavevmode
If $d = b$ or $d = b - 1$, then both statements are immediate from the formula in Lemma~\ref{easycontlemma}. (Note that $c(d)$ is always 1.)  Assume therefore that $d \leq b - 2$. 
\begin{enumerate}
\item If $d = D$, then the inequality is trivial; so assume $D < d$. By equation \eqref{eq:carryr}, we have $$\textstyle c(d) - c(d - D) = c(D) - R(d, D).$$
Certainly $r_b(\frac{D}{d}, \frac{d-D}{d})$ can be no greater than $0.1^\infty$. Therefore $$R(d, D) = D \reed_\beta r_b(\textstyle \frac{D}{d}, \frac{d-D}{d}) \leq D \reed_\beta(0.1^\infty) = \displaystyle \frac{D}{\beta - 1}.$$
On the other hand, by Proposition \ref{contDlemma}, we know that $c(D) \geq \frac{D (\beta + 1)}{\beta (\beta - 1)} > \frac{D}{\beta - 1}$.  Therefore $c(d) > c(d- D)$ (and in fact the inequality is strict). 
\item It suffices to show that, for $0 < i < d$, we have $c(i)  >  c(i - 1)$. By equation \eqref{eq:carryr} this is equivalent to the inequality $c(1) > R(i, 1)$. Since $i < d$, we know that $$R(i, 1) = D \reed_\beta r_b(\textstyle \frac{i-1}{d}, \frac{1}{d}) \leq D \reed_\beta(0.01^\infty) = \displaystyle \frac{D}{\beta (\beta - 1)}.$$
Now Corollary \ref{cor:cont1} completes the claim.
\end{enumerate}
\end{proof}

\begin{lemma}[Step property]\label{lemma:step}
Suppose $d = b$ or $d = b - 1$ or $D \leq \frac{b}{2}$. If $(i, j) \in \mathcal I$, and $n, m$ are integers with $db^k \leq n < db^{k + 1}$ and $j b^k \leq m$, then 
$$c(n) - c(n - i b^k) \geq c(m) - c(m - j b^k).$$
\end{lemma}
Here as before $\mathcal I  = \{(i, j) : 0 \leq j < j + D \leq i \leq d\} \cup \{(d,d)\}$ is the set of pairs $(i, j)$ so that $y^j X^{d - i}$ can appear in the companion polynomial of the recursion in question; see proof of Proposition \ref{nilgrowthprop}.

\begin{proof}
We  use Lemma \ref{lemma:scalingcontent} to divide each equation by $\beta^k$, and replace $n$ and $m$ by $\frac{n}{b^k}$ and $\frac{m}{b^k}$, respectively. We therefore seek to show that for $n, m \in \Z[\frac{1}{b}]_{\geq 0}$ satisfying $d \leq n < db$ and $m \geq j$, we have 
$$c(n) - c(n - i) \stackrel{?}\geq c(m) - c(m- j).$$
Using equation \eqref{eq:carryr} twice, the desired statement is equivalent to 
$$c(i) - R(n, i) \stackrel{?} \geq c(j) - R(m, j).$$
Or, equivalently, 
$$c(i) - c(j) \stackrel{?}\geq R(n, i) - R(m, j).$$
Since $R(m, j)$ is nonnegative, it suffices to show that 
$$c(i) - c(j) \stackrel{?}\geq R(n, i).$$
We now take two cases. If $i = d$, then $R(n, i) = 0$, so that it suffices to show that $c(d) \geq c(j)$ for $(d, j) \in \mathcal I$. For $j = d$, this is clear; and for $j \leq i - D$, this follows from both parts of Lemma~\ref{lemma:base} above. 

If, on the other hand, $i < d$, then Equation \eqref{eq:carryr} gives us $c(i) - c(j) = c(i - j) - R(i,j)$, which reduces the desired statement to 
\begin{equation}\label{last}
c(i-j) \stackrel{?}\geq R(n, i) + R(i, j).
\end{equation}

If $d = b$ or $d = b - 1$, then $R(i, j) = 0$; and (by Lemma \ref{easycontlemma}) we have $c(i - j) = \frac{i-j}{d - D} \geq \frac{D}{d-D}$, so it remains to show that $R(n, i) \stackrel{?}\leq \frac{D}{d-D}$. If $d = b$, then at most one digit is carried, so that $R(n, i) \leq D \reed_\beta(0.1) = \frac{D}{\beta}$. And if $d = b-1$, then every digit may be carried, so that $R(n, i) \leq D \reed_\beta(0.1^\infty) = \frac{D}{\beta - 1}$. In both cases, the desired inequality holds.  

On the other hand if $d \leq b - 2$ (and so $D \leq \frac{b}{2}$), then we reason as follows. The left-hand side of desired inequality \eqref{last} is bounded below by $c(D)$, and the right-hand side is bounded above by $$D \reed_\beta(0.1^\infty) + D \reed_\beta (0.01^\infty) = \frac{D (\beta + 1)}{\beta(\beta - 1)}.$$ 
Therefore it suffices to show that $c(D) \stackrel?\geq \frac{D (\beta + 1)}{\beta(\beta - 1)},$
which is established in Proposition \ref{contDlemma}.\end{proof} 

Lemmas \ref{lemma:base} and \ref{lemma:step} complete the proof of Theorem \ref{witnessthm}, which in turn completes the proof of Theorem \ref{specialngtthm}, and hence of Theorem \ref{ngtthm}.

\section{Complements}\label{complementsec}

\subsection{Refinement of Theorem \ref{toyngtthm}}
We state a refinement of the toy version of the NGT (Theorem \ref{toyngtthm}). One can also obtain similar refinements of Theorem \ref{specialngtthm}.

\begin{thm}[Refined toy NGT]\label{nilpyd1}
Let $\F$ be a field of characteristic $p$ and let $q = p^k$. Suppose that $T: \F[y] \to \F[y]$ is an $\F$-linear operator satisfying the following two conditions:  
\begin{enumerate}
\item For $f \in \F[y]$, we have $\deg T(f) \leq \deg f - E$ for some $E \geq 1$. 
\item The sequence $\{T(y^n)\}_n$ satisfies a linear recursion whose companion polynomial has the shape 
$$P = (X + c y)^d + (\mbox{terms of total degree} \leq d - D) \in \F[y][X]$$
for some $d \leq q$, some $D \geq 1$, and some $c \in \F$.  
\end{enumerate}
Then $$N_T(f) \leq \frac{(q-D)(q - 1)}{E (q - D - 1)}(\deg f)^\frac{\log (q - D)}{\log q}.$$
\end{thm}

\begin{rk}
Given the shape of the companion polynomial of the recursion (i.e., the total degree of the companion polynomial is the same as the order of the recursion), it suffices to check the condition that $\deg T(f) \leq \deg f - E$ on $f = 1, y, \ldots, y^{d-1}$ only.
\end{rk} 

\begin{proof}
The sequence $\{T(y^n)\}_n$ also satisfies the linear recursion with companion polynomial 
$$P' = (X + cy)^{q-d} P = X^q + cy^{q} + (\mbox{terms of total degree}\leq q - D).$$
Let $c = c_{q, q-D}$, define $\cc: \F[y] \to \N \cup \{-\infty\}$ from $c$ as in the proof of Theorem \ref{toyngtthm}, and follow the same inductive argument \emph{mutatis mutandis} to show that $\cc(Tf) \leq \cc(f) - E$. (The main adjustment is in Proposition \ref{babycontent}\eqref{iDitem}: if $i$ is a digit base $q$ and $n \geq i$ has no more than $2$ digits base $q$, then $c(n) - c(n - i)$ is either $i$ or $i - D$; see  Lemma~\ref{lemma:carrydigits} for a conceptual explanation.) 

We have shown therefore shown that $N_T(y^n) \leq \frac{c(n)}{E}$. Lemma \ref{contentgrowthlemma} completes the proof.
\end{proof}

\subsection{Comments on $\alpha$ in Theorem \ref{specialngtthm}}
How optimal is the order of growth of the nilpotence index $\alpha$ from the empty-middle NGT? 

To this end, if $K$ is a field, and $T: K[y] \to K[y]$ a degree-lowering linear operator, let 
$$\alpha(T) := \limsup_{n \to \infty} \frac{\log N_T(y^n)}{\log n},$$
and let $$\alpha_K(d, D) = \sup_{T \in \mathcal L_k(d, D)} \{ \alpha(T) \},$$
where $\mathcal L_k(d, D)$ is the set of degree-lowering operators $T: K[y] \to K[y]$ with $\{T(y^n)\}_n$ satisfying a recurrence with companion polynomial $X^d + c y^d + (\mbox{terms of total degree} \leq d - D)$ for some $c \in K$. Since $N_T(y^n) \leq n$, we know that $\alpha_K(d, D) \leq 1$. The following proposition clarifies that studying $\alpha_K(d, D)$ is only interesting in characteristic $p$. 

\begin{prop}\label{prop:charzeroex}
If $K$ has characteristic zero and $D < d$, then $\alpha_K(d, D) = 1$. 
\end{prop}

\begin{proof}
Fix $d$, and consider the recursion operator $T:K[y] \to K[y]$ defined by the companion polynomial $P = X^d - y^d - y$, corresponding to the recurrence $T(y^n) = (y^d + y) T(y^{n-d})$, and initial values $\{T(y^n)\}_{n = 0}^{d-1} = \{0, \ldots, 0, 1\}$. We will show that $N_T(y^{kd + d - 1}) = \lfloor\frac{k}{d-1}\rfloor + 1$, which will establish that $\alpha_T = 1$. 

Indeed, from the recurrence, we have  $T(y^n) = 0$ if $n \not \equiv -1 \cmod{d}$, and $T(y^{kd + d - 1}) =(y^{d} + y)^{k}$.  For $f = \sum a_n y^n \in K[y]$, write $e(f)$ for the set $\{n : a_n \neq 0\}$ of exponents appearing in $f$. From above, we see that 
$$e( T (y^{kd + d - 1})) = \{k, k + (d-1), k + 2(d-1), \ldots, kd\}.$$ 
More generally, we can show by induction that the set $S_{m, k} := e\big(T^m(y^{kd + d -1})\big)$ is an arithmetic progression of common difference $d-1$, greatest term $d\big(k - (m-1)(d-1)\big)$, and length $k - (m-1)(d-1) + 1$, so long as $k \geq (m-1)(d-1)$; otherwise the set is empty and $T^m(y^{kd + d -1}) = 0$. Indeed, from the explicit formulation of $T(y^n)$, we see that if $S_{m, k}$ is as claimed, then the greatest element of $S_{m + 1, k} = N - (d-1)$, where $N$ is the greatest element of $S_{m, k}$ congruent to $d-1$ modulo $d$. Since the maximum element of $S_{m, k}$ is congruent to $0$ modulo $d$, and every successive smaller element is $d-1$ less, we see that $N$ is the $d^{\rm th}$ greatest element of $S_{m, k}$. In other words, $$N = d\big(k - (m-1)(d-1)\big) - (d-1)^2,$$ so that the greatest element of $S_{m+1, k}$ is $N - (d-1) =  d\big(k - m(d-1)\big),$ as desired. 
Since the relevant coefficients are positive and we are in infinite characteristic, no cancelation of intermediate terms is possible. Finally, since $T^m(y^{kd + d - 1}) \neq 0$ if and only if $k \geq (m-1)(d-1)$, we have $N_T(y^{kd + d - 1}) = \lfloor\frac{k}{d-1}\rfloor + 1$, as claimed. 
\end{proof}
 
In characteristic $p$, let us confine our inquiry to the case where $d$ can be taken to be a power of $p$, as in Theorem \ref{nilpyd1} above. Theorem \ref{nilpyd1} tells us that $\alpha_{\F}(p^k, D) \leq \frac{\log(p^k - D)}{\log p^k}$. How optimal is this estimate? Computationally, it appears that for $k = 1$ this inequality is optimal. A few examples for $D = 1$:

\begin{example}

\begin{itemize}
\item $p = 3$: The recursion operator $T$ with companion polynomial $X^3 + yX - y^3$ and initial values $\{0, 1, y\}$ appears to achieve $N_T(y^n) = c_{3, 2}(n)$ infinitely often. 
\item $p = 5$: The recursion operator $T$ with companion polynomial $X^5 + 3yX^3 + y^2X^2 + 3y^3X + 4y^5$ and initial values $[0, 1, y, y^2, y^3]$ appears to achieve $N_T(y^n) = c_{5, 4}(n)$ for ``most'' $n$: every counterexample $n$ has $0$s in its base-$5$ expansion.
\item $p = 7$: The recursion operator $T$ with companion polynomial $$X^7 + 3 y^2 X^4 + 6 y^3 X^3 + 5 y^4 X^2 + 3 y^5 X + 6 y^7$$ appears to achieve $N_T(y^n) = c_{7, 6}(n)$ for most $n$. For $n < 1000$, there are only 36 counterexamples, and $c_{7, 6}(n) - N_T(y^n) \leq 3$ for each one.  
\item $p = 11$. The recursion operator $T$ with companion polynomial $$P_T = X^{11} + 6 y X^9 + 2 y^2 X^8 + 3 y^3 X^7 + 6 y^4 X^6 + 8 y^6 X^4 + y^8 X^2 + 9 y^9 X + 10 y^{11}$$ appears to achieve $N_T(y^n) = c_{11, 10}(n)$ for most $n$. For $n < 1000$, there are only 8 counterexamples, and $N_T(y^n) = c_{11, 10}(n) - 1$ for each one. 
 \end{itemize}
\end{example}

The inequality does not appear to be optimal already for $p^k$ with $k \geq 2$. Further study of this behavior awaits.

\pagestyle{plain}
\bibliographystyle{acm}
{\bibliography{modformsmodp}

\begin{thebibliography}{10}

\bibitem{kamal}
{\sc {Al Hajj Shehadeh}, H., Jaafar, S., and Khuri-Makdisi, K.}
\newblock Generating functions for {H}ecke operators.
\newblock {\em International Journal of Number Theory 5}, 1 (2009), 125--140.
\newblock Available at \url{https://doi.org/10.1142/S1793042109001979}.

\bibitem{alloucheshallit}
{\sc Allouche, J.-P., and Shallit, J.}
\newblock The ring of {$k$}-regular sequences.
\newblock {\em Theoretical Computer Science 98\/} (1992), 163--197.
\newblock Available at \url{https://doi.org/10.1016/0304-3975(92)90001-V}.

\bibitem{Bpseudodef}
{\sc Bella{\"{\i}}che, J.}
\newblock Pseudodeformations.
\newblock {\em Mathematische Zeitschrift 270}, 3-4 (2012), 1163--1180.

\bibitem{BK}
{\sc Bella\"{i}che, J., and Khare, C.}
\newblock Level 1 {H}ecke algebras of modular forms modulo $p$.
\newblock {\em Compositio Mathematica 151}, 3 (2015), 397--415.
\newblock Available at \url{https://doi.org/10.1112/S0010437X1400774X}.

\bibitem{BuzzCale:slopes}
{\sc Buzzard, K., and Calegari, F.}
\newblock Slopes of overconvergent 2-adic modular forms.
\newblock {\em Composition Math. 141\/} (2005), 591--604.
\newblock Available at \url{https://doi.org/10.1112/S0010437X04001034}.

\bibitem{kconrad:recursion}
{\sc Conrad, K.}
\newblock Solving linear recursions over all fields.
\newblock Available at
  \url{http://www.math.uconn.edu/~kconrad/blurbs/linmultialg/linearrecursion.pdf}.

\bibitem{deo}
{\sc Deo, S.}
\newblock Level ${N}$ {H}ecke algebras of modular forms modulo $p$.
\newblock {\em Algebra and Number Theory 11}, 1 (2017), 1--38.
\newblock Available at \url{http://msp.org/ant/2017/11-1/p01.xhtml}.

\bibitem{emertonBBK}
{\sc Emerton, M.}
\newblock {$p$}-{A}dic families of modular forms (after {H}ida, {C}oleman, and
  {M}azur).
\newblock {\em Ast\'erisque}, 339 (2011), Exp. No. 1013, vii, 31--61.
\newblock S{\'e}minaire Bourbaki. Vol. 2009/2010. Expos{\'e}s 1012--1026.

\bibitem{mathilde}
{\sc Gerbelli-Gauthier, M.}
\newblock The order of nilpotence of {H}ecke operators mod 2: a new proof.
\newblock {\em Research in Number Theory 2}, 1 (2016).
\newblock Available at \url{https://doi.org/10.1007/s40993-016-0038-6}.

\bibitem{infinitefern}
{\sc Gouv{\^e}a, F.~Q., and Mazur, B.}
\newblock On the density of modular representations.
\newblock In {\em Computational perspectives on number theory ({C}hicago, {IL},
  1995)}, vol.~7 of {\em AMS/IP Stud. Adv. Math.} Amer. Math. Soc., Providence,
  RI, 1998, pp.~127--142.

\bibitem{JochStudy}
{\sc Jochnowitz, N.}
\newblock A study of the local components of the {H}ecke algebra mod {$l$}.
\newblock {\em Transactions of the American Mathematical Society 270}, 1
  (1982), 253--267.

\bibitem{medved:heckedim}
{\sc Medvedovsky, A.}
\newblock Lower bounds on dimensions of mod-$p$ {H}ecke algebras in the
  genus-zero case.
\newblock In preparation.

\bibitem{medved}
{\sc Medvedovsky, A.}
\newblock Lower bounds on dimensions of mod-$p$ {H}ecke algebras: The
  nilpotence method.
\newblock Ph.D. thesis, 2015. Available at
  \url{http://www.math.brown.edu/~medved/Mathwriting/DissertationMedvedovsky_Fall2015.pdf}.

\bibitem{NS1}
{\sc Nicolas, J.-L., and Serre, J.-P.}
\newblock Formes modulaires modulo 2 : l'ordre de nilpotence des op\'{e}rateurs
  de {H}ecke modulo 2.
\newblock {\em Comptes rendus math\'{e}matique. Acad\'{e}mie des Sciences.
  Paris 350\/} (2012).

\bibitem{NS2}
{\sc Nicolas, J.-L., and Serre, J.-P.}
\newblock Formes modulaires modulo 2 : structure de l'alg\`{e}bre de {H}ecke.
\newblock {\em Comptes rendus math\'{e}matique. Acad\'{e}mie des Sciences.
  Paris 350\/} (2012).

\bibitem{Smod3}
{\sc Serre, J.-P.}
\newblock {\em {\OE}uvres. {V}ol. {III}}.
\newblock Springer-Verlag, Berlin, 1986, p.~710.
\newblock Note 229.2.

\bibitem{SwDy}
{\sc Swinnerton-Dyer, H. P.~F.}
\newblock On {$l$}-adic representations and congruences for coefficients of
  modular forms.
\newblock In {\em Modular functions of one variable, {III} ({P}roc. {I}nternat.
  {S}ummer {S}chool, {U}niv. {A}ntwerp, 1972)}. Springer, Berlin, 1973,
  pp.~1--55. Lecture Notes in Math., Vol. 350.

\bibitem{Tmod2}
{\sc Tate, J.}
\newblock The non-existence of certain {G}alois extensions of {${\mathbb Q}$}
  unramified outside {$2$}.
\newblock In {\em Arithmetic geometry ({T}empe, {AZ}, 1993)}, vol.~174 of {\em
  Contemp. Math.} Amer. Math. Soc., Providence, RI, 1994, pp.~153--156.

\bibitem{sage}
{\sc {The Sage Developers}}.
\newblock {\em {S}ageMath, the {S}age {M}athematics {S}oftware {S}ystem
  ({V}ersion 7.3)}, 2016.
\newblock \url{http://www.sagemath.org}.

\end{thebibliography}

\end{document}